\documentclass[12pt]{amsart}
\usepackage{amsmath,amsthm,amsfonts,amscd,amssymb,eucal,latexsym,mathrsfs, appendix, yhmath, setspace}
\usepackage[numbers,sort&compress]{natbib}
\usepackage[all,cmtip]{xy}
\usepackage{enumerate}

\newtheorem{theorem}{Theorem}[section]
\newtheorem{corollary}[theorem]{Corollary}
\newtheorem{lemma}[theorem]{Lemma}
\newtheorem{proposition}[theorem]{Proposition}

\theoremstyle{definition}
\newtheorem{definition}[theorem]{Definition}
\newtheorem{remark}[theorem]{Remark}

\newtheorem{example}[theorem]{Example}
\newtheorem{question}[theorem]{Question}

\newcommand{\rad}{{\rm rad}}

\newcommand{\im}{{\rm im}}

\newcommand{\diag}{{\rm diag}}

\newcommand{\cA}{{\mathcal A}}

\newcommand{\cG}{{\mathcal G}}
\newcommand{\cH}{{\mathcal H}}

\newcommand{\cJ}{{\mathcal J}}

\newcommand{\cM}{{\mathcal M}}
\newcommand{\cN}{{\mathcal N}}

\newcommand{\cW}{{\mathcal W}}

\newcommand{\Cb}{{\mathbb C}}
\newcommand{\Db}{{\mathbb D}}

\newcommand{\Kb}{{\mathbb K}}

\newcommand{\Nb}{{\mathbb N}}

\newcommand{\Pb}{{\mathbb P}}

\newcommand{\Rb}{{\mathbb R}}

\newcommand{\Zb}{{\mathbb Z}}

\newcommand{\spec}{{\rm spec}}

\newcommand{\supp}{{\rm supp}}

\newcommand{\rM}{{\rm M}}

\newcommand{\rk}{{\rm rk}}

\newcommand{\rC}{{\rm C}}

\newcommand{\rW}{{\rm W}}
\newcommand{\rU}{{\rm U}}
\newcommand{\rV}{{\rm V}}
\newcommand{\sFPM}{{\mathscr{M}_{\rm fp}}}
\newcommand{\sFGM}{{\mathscr{M}_{\rm fg}}}

\allowdisplaybreaks

\begin{document}

\title{Malcolmson semigroups}

\author{Tsz Fun Hung}
\address{Tsz Fun Hung,
Department of Mathematics, SUNY at Buffalo, Buffalo, NY 14260-2900, USA
}
\address{Current address: Department of Mathematics, Fort Lewis College, 1000 Rim Drive, Durango, CO 81301, USA
}
\email{thung@fortlewis.edu}

\author{Hanfeng Li}
\address{Hanfeng Li,
Department of Mathematics, SUNY at Buffalo, Buffalo, NY 14260-2900, USA
}
\email{hfli@math.buffalo.edu}

\date{January 22, 2023}


\begin{abstract}
Inspired by the construction of the Cuntz semigroup for a $C^*$-algebra, we introduce the matrix Malcolmson semigroup and the finitely presented module Malcolmson semigroup for a unital ring. These two semigroups are shown to have isomorphic Grothendieck group in general and be isomorphic for von Neumann regular rings. For unital $C^*$-algebras, it is shown that the matrix Malcolmson semigroup has a natural surjective order-preserving homomorphism to the Cuntz semigroup, every dimension function is a Sylvester matrix rank function, and there exist Sylvester matrix rank functions which are not dimension functions.
\end{abstract}

\maketitle

\tableofcontents

%%%%%%%%%%%%%%%%%%%%%%%%%%%%%%%%%%%%%%%%%%%%%%%%%%%%%%%%%%%%%%%%%%%%%%%%%%%%%%%%%%%%%%%%%%%%%%%%%%%%
\section{Introduction} \label{S-introduction}

The Cuntz semigroup and dimension function for $C^*$-algebras were introduced by Cuntz \cite{Cuntz78} as means to study simple $C^*$-algebras. For unital $C^*$-algebras $\cA$, the Cuntz semigroup $\rW_\rC(\cA)$ is a partially ordered abelian monoid with a fixed order-unit and the dimension functions are suitable $\Rb_{\ge 0}$-valued functions defined on the space $M(\cA)$ of rectangular matrices over $\cA$ and can be identified with the states of the Cuntz semigroup, i.e. semigroup homomorphisms $\rW_\rC(\cA)\rightarrow \Rb_{\ge 0}$ preserving the order and order-unit. Toms used the Cuntz semigroup to distinguish two simple separable nuclear unital $C^*$-algebras with the same Elliott invariants \cite{Toms}. Ever since then the Cuntz semigroup has played a vital role in the Elliott classification program \cite{APT, APRT, BPT, CEI, Thiel, Winter}.

The Sylvester rank functions for unital rings were introduced by Malcolmson \cite{Malcolmson80}. It can be defined equivalently at several levels: at the matrix level as suitable $\Rb_{\ge 0}$-valued functions on the space $M(R)$ of rectangular matrices over $R$ \cite{Malcolmson80}, at the module level as suitable $\Rb_{\ge 0}$-valued functions on the class $\sFPM(R)$ of finitely presented left $R$-modules \cite{Malcolmson80}, at the homomorphism level as
suitable $\Rb_{\ge 0}$-valued functions on the class of homomorphisms between finitely generated projective left $R$-modules \cite{Schofield}, and at the pair of modules level as suitable $\Rb_{\ge 0}\cup \{+\infty\}$-valued functions on the space of pairs of left $R$-modules $\cM_1\subseteq \cM_2$ \cite{Li21}.

The Sylvester rank functions have various applications in ring theory and $L^2$-invariants theory. The classical result of Cohn on epimorphisms from $R$ to division rings \cite{Cohn71} can be stated as that there is a natural $1$-$1$ correspondence between $\Zb_{\ge 0}$-valued Sylvester rank functions on $R$ and isomorphism classes of such homomorphisms \cite{Malcolmson80}. Schofield further showed that for algebras $R$ over a field, there is a natural $1$-$1$ correspondence between $\frac{1}{n}\Zb_{\ge 0}$-valued Sylvester rank functions on $R$ for some $n\in \Nb$  and equivalence classes of homomorphisms from $R$ to simple Artinian rings \cite[Theorem 7.12]{Schofield}, where two such homomorphisms $\varphi_1: R\rightarrow S_1$ and $\varphi_2:R\rightarrow S_2$ are called equivalent if there are a simple Artinian ring $S$ and  homomorphisms $\psi_i: S_i\rightarrow S$ for $i=1, 2$ with $\psi_1\circ \varphi_1=\psi_2\circ \varphi_2$. Ara, O'Meara, and Perera used Sylvester rank functions to establish Kaplansky's direct-finiteness conjecture for free-by-amenable groups \cite{AOP}, which was later extended by Elek and Szab\'{o} to all sofic groups \cite{ES04}. Sylvester rank functions also played a vital role in the recent work of Jaikin-Zapirain and L\'{o}pez-\'{A}lvarez on the Atiyah conjecture and the L\"{u}ck approximation conjecture \cite{JZ19, JL20}.  See also \cite{AC20, AC22a, AC22b, DS, Elek17, JZ17, JL21, JL22, Virili19}.

For a unital $C^*$-algebra $\cA$, the most important dimension functions and Sylvester matrix rank functions are induced from tracial states. In fact, when $\cA$ is exact, all lower semicontinuous dimension functions of $\cA$ are induced from tracial states of $\cA$ \cite{BH, Haagerup}.
In this work we give a systematic study of the relation between dimension functions and Sylvester matrix rank functions for $\cA$. It turns out that every dimension function is a Sylvester matrix rank function (see Proposition~\ref{P-M to C}), while not every Sylvester matrix rank function is a dimension function (Examples~\ref{E-rk not dim} and \ref{E-rk not dim2}). On the other hand, certain nice properties of the dimension functions still hold for all Sylvester matrix rank functions (Proposition~\ref{P-absolute value}).  From any Sylvester matrix rank function for $\cA$, one can construct a lower semicontinuous dimension function (Proposition~\ref{P-lower dimension}). In particular, if $\cA$ is exact and has a Sylvester rank function, then $\cA$ has a tracial state (Corollary~\ref{C-trace}).

The analogy between dimension functions and Sylvester matrix rank functions motivates us to introduce an algebraic counterpart of the Cuntz semigroup, called the matrix Malcolmson semigroup and denoted by $\rW_\rM(R)$, for any unital ring $R$. This is also a partially ordered abelian monoid with a fixed order-unit. For a unital $C^*$-algebra $\cA$, there is a natural surjective order-preserving homomorphism $\rW_\rM(\cA)\rightarrow \rW_\rC(\cA)$ (Proposition~\ref{P-M to C}), which may fail to be injective (Remark~\ref{R-M to C}).

As the Sylvester rank functions can be defined at both the matrix and the module levels, we also introduce a finitely presented module version of $\rW_\rM(R)$, denoted by $\rV_\rM(R)$, for any unital ring $R$. The Grothendieck groups of $\rW_\rM(R)$ and $\rV_\rM(R)$ are naturally isomorphic (Theorem~\ref{T-matrix vs fp}). When $R$ is von Neumann regular, $\rW_\rM(R)$ and $\rV_\rM(R)$ are the same (Theorem~\ref{T-regular}).

We take one step further to introduce a finitely generated module version of $\rV_\rM(R)$, denote by $\rU_\rM(R)$. It is somehow surprising that the partial order on $\rU_\rM(R)$ extends that on $\rV_\rM(R)$ (Theorem~\ref{T-fp vs fg}).

A fundamental result in the theory of partially ordered abelian groups is the Goodearl-Handelman theorem \cite{GH76, Goodearl86} describing the range of an element under states extending a given state of a subgroup. In applications to the study of partially ordered abelian semigroups such as $\rW_\rC(\cA)$ and $\rW_\rM(R)$, one needs to pass to the Grothendieck group to apply the Goodearl-Handelman theorem. Recently Antoine, Perera, and Thiel give a nice version of the Goodearl-Handelman theorem for positively ordered monoids \cite{APT18}.
In an appendix we give a version in the setting of
partially ordered abelian semigroups (Theorem~\ref{T-extension submonoid}), unifying the theorems of Goodearl-Handelman and Antoine-Perera-Thiel.
This enables us to show that a certain bound for the range of dimension function given by Cuntz in \cite{Cuntz78} is in fact precisely the range (Remark~\ref{R-Cuntz}) and to construct some interesting Sylvester rank functions in Theorem~\ref{T-rk for square} and Examples~\ref{E-rk not dim}.
Though these applications can be obtained using the Antoine-Perera-Thiel result for positively ordered monoids, the extension to partially ordered abelian semigroups is of some interest itself.

This paper is organized as follows. We collect some basic facts about Grothendieck groups, partially ordered abelian semigroups, Cuntz semigroups, and Sylvester rank functions in Section~\ref{S-prelim}.
In Section~\ref{S-Matrix M} we introduce $\rW_\rM(R)$ and calculate it for left Artinian unital local rings whose maximal ideal is generated by a central element.  Some interesting Sylvester rank functions are constructed in Section~\ref{S-commutative}. We compare $\rW_\rC(\cA)$ and dimension functions with $\rW_\rM(\cA)$ and Sylvester matrix rank functions for a unital $C^*$-algebra $\cA$ in Section~\ref{S-M vs C}. In Section~\ref{S-fp Module M} we introduce $\rV_\rM(R)$ and study its relation with $\rW_\rM(R)$. In Section~\ref{S-fg Module M} we introduce $\rU_\rM(R)$ and show that $\rV_\rM(R)$ embeds into $\rU_\rM(R)$. The extension of the Goodearl-Handelman theorem and the Antoine-Perera-Thiel theorem to partially ordered abelian semigroups is given in Appendix~\ref{S-POAS}.

\noindent{\it Acknowledgements.}
Both authors were partially supported by NSF grant DMS-1900746. We are grateful to Leonel Robert for bringing the paper \cite{Fuchssteiner} to our attention and to Francesc Perera and Hannes Thiel for comments. We also thank the referees for very helpful comments.

%%%%%%%%%%%%%%%%%%%%%%%%%%%%%%%%%%%%%%%%%%%%%%%%%%%%%%%%%%%%%%%%%%%%%%%%%%%%%%%%%%%%%%%%%%%%%%%%%%%%%%%%%%%%%%%%%%%%%%%%%%%%%%%%%%%%%%%%%%%%%%%%%%%%%%%%%%
\section{Preliminaries} \label{S-prelim}

In this section we recall some basic definitions and facts about Grothendieck groups, partially ordered abelian semigroups, Cuntz semigroups, and Sylvester rank functions. Throughout this article, all unital rings are assumed to be nonzero.

%%%%%%%%%%%%%%%%%%%%%%%%%%%%%%%%%%%%%%%%%%%%%%%%%%%%%%%%%%%%%%%%%%%%%%%%%%%%%%%%%%%%%%%%%%%%%%%%%%%%%%%%%%%%%%%%%%%%%%%%%%%%%%%%%%%%%%%%%%%%%%%%%%%%%%
\subsection{Grothendieck groups, partially ordered abelian semigroups, and states} \label{SS-state}

Let $\rW$ be an abelian semigroup. Denote by $\cG(\rW)$ the Grothendieck group of $\rW$. Then there is a canonical semigroup homomorphism $\rW\rightarrow \cG(\rW)$.  For each $a\in \rW$ denote by $[a]$ the image of $a$ in $\cG(\rW)$. Then the elements of $\cG(\rW)$ are all of the form $[a]-[b]$ for $a, b\in \rW$, and $[a]=[b]$ exactly when $a+c=b+c$ for some $c\in \rW$.

Now assume further that $(\rW, \preceq)$ is a {\it partially ordered abelian semigroup} in the sense that $\rW$ is an abelian semigroup, and $\preceq$ is a partial order on $\rW$ such that $a+c\preceq b+c$ whenever $a\preceq b$ and $c\in \rW$. Denote by $\cG(\rW)_+$ the set of elements of $\cG(\rW)$ of the form $[a]-[b]$ such that $b+c\preceq a+c$ for some $c\in \rW$.

Lemmas~\ref{L-positive cone} and \ref{L-scaled ordered} below are well known. We include a proof for convenience of the reader.

\begin{lemma} \label{L-positive cone}
Let $a, b\in \rW$. Then $[a]-[b]\in \cG(\rW)_+$ if and only if $b+c\preceq a+c$ for some $c\in \rW$.
\end{lemma}
\begin{proof} The ``if'' part follows from the definition of $\cG(\rW)_+$.

Assume that $[a]-[b]\in \cG(\rW)_+$. Then $[a]-[b]=[c]-[d]$ for some $c, d\in \rW$ such that $d+u\preceq c+u$ for some $u\in \rW$. Since $$[a+d]=[a]+[d]=[b]+[c]=[b+c],$$ we have $a+d+v=b+c+v$ for some $v\in \rW$.  Then
$$ b+(c+v+u)=a+d+v+u\preceq a+(c+v+u).$$
This proves the ``only if'' part.
\end{proof}

For a partially ordered abelian semigroup $(\rW, \preceq)$, we say $v\in \rW$ is an {\it order-unit} if for any $a\in \rW$ there is some $n\in \Nb$ such that $a\preceq a+v$, $a\preceq nv$, and $v\preceq a+nv$ \cite[Definition 2.1]{BR92}.  For a partially ordered abelian semigroup $(\rW, \preceq)$ with order-unit $v$, we denote it by $(\rW, \preceq, v)$.

For a partially ordered abelian group $(\cG, \preceq)$, set $\cG_+=\{a\in \cG: 0\preceq a\}$, where $0$ is the identity element of $\cG$. Then $\cG_+$ is a subsemigroup of $\cG$
and $\cG_+\cap (-\cG_+)=\{0\}$. Conversely, given any abelian group $\cG$ and any subsemigroup $\cG_+$ of $\cG$ satisfying $\cG_+\cap (-\cG_+)=\{0\}$, we have the partial order $\preceq$ on $\cG$ defined by $a\preceq b$ if $b-a\in \cG_+$. Thus we also write $(\cG, \cG_+)$ for a partially ordered abelian group.

For any partially ordered abelian group $(\cG, \cG_+)$, note that $u\in \cG$ is an order-unit exactly if  $u\in \cG_+$ and for any $a\in \cG$ one has $a\preceq nu$ for some $n\in \Nb$. In such case $\cG\neq \{0\}$ if and only if $u\neq 0$, where $0$ is the identity element of $\cG$.

\begin{lemma} \label{L-scaled ordered}
Let $(\rW, \preceq)$ be a partially ordered abelian semigroup with order-unit $v$. Then $(\cG(\rW), \cG(\rW)_+)$ is a partially ordered abelian group with order-unit $[v]$.
\end{lemma}
\begin{proof} Let $[a]-[b], [c]-[d]\in \cG(\rW)_+$. Then $b+x \preceq a+x$ and $d+y \preceq c+y$ for some $x, y\in \rW$. Thus
$$([a]-[b])+([c]-[d])=[a+c]-[b+d]$$ with
$$b+d+(x+y)\preceq a+c+(x+y),$$ and hence $([a]-[b])+([c]-[d])\in \cG(\rW)_+$. Therefore $\cG(\rW)_+$ is a subsemigroup of $\cG(\rW)$.

We have $0=[a]-[a]\in \cG(\rW)_+\cap (-\cG(\rW)_+)$ for any $a\in \rW$.

Let $[a]-[b]\in \cG(\rW)_+\cap (-\cG(\rW)_+)$. By Lemma~\ref{L-positive cone} we have $b+c\preceq a+c$ and $a+d\preceq b+d$ for some $c, d\in \rW$.
Then $b+c+d\preceq a+c+d$ and $a+c+d\preceq b+c+d$. Thus $b+c+d=a+c+d$. Therefore $[a]-[b]=0$. This shows that $\cG(\rW)\cap (-\cG(\rW)_+)=\{0\}$.

We conclude that $(\cG(\rW), \cG(\rW)_+)$ is a partially  ordered abelian group.

Note that $[v]=[v+a]-[a]\in \cG(\rW)_+$ for any $a\in \rW$. Let $[a]-[b]\in \cG(\rW)_+$. Then there are some $n, m\in \Nb$ such that $v\preceq b+nv$ and $a\preceq mv$. It follows that $a\preceq mv\preceq b+(n+m-1)v$, whence $[a]-[b]\preceq (n+m-1)[v]$. Therefore $[v]$ is an order-unit of $(\cG(\rW), \cG(\rW)_+)$.
\end{proof}

For a partially ordered abelian semigroup $(\rW, \preceq, v)$, a {\it state} of $(\rW, \preceq, v)$ is a semigroup homomorphism $\varphi: \rW\rightarrow \Rb$ such that $\varphi(v)=1$ and $\varphi(a)\le \varphi(b)$ whenever $a\preceq b$ \cite[Definition 2.1]{BR92}. Clearly there is a natural one-to-one correspondence between states of $(\rW, \preceq, v)$ and states of $(\cG(\rW), \cG(\rW)_+, [v])$.

The following is the fundamental result of Goodearl and Handelman on extensions of states of partially ordered abelian groups with order-unit \cite[Lemma 3.1 and Theorem 3.2]{GH76} \cite[Lemma 4.1 and Proposition 4.2]{Goodearl86}.

\begin{theorem} \label{T-GH}
Let $(\cG, \cG_+, u)$ be a partially ordered abelian group with order-unit.
Let $\cH$ be a subgroup  of $\cG$ containing $u$, and let $\varphi$ be a state of $(\cH, \cH\cap \cG_+, u)$. Let $a\in \cG$.  Set
$$p=\sup\{\varphi(b)/m \mid b\in \cH, m\in \Nb, b\preceq ma\},$$ and
$$q=\inf\{\varphi(b)/m \mid b\in \cH, m\in \Nb, b\succeq ma\}.$$ Then $-\infty<p\le q<\infty$, and
$$[p, q]=\{\psi(a)\mid  \psi \mbox{ is a state of } (\cG, \cG_+, u) \mbox{ extending } \varphi\}.$$ In particular, every state of $(\cH, \cH\cap \cG_+, u)$ extends to a state of $(\cG, \cG_+, u)$.
\end{theorem}

A partially ordered abelian semigroup $(\rW, \preceq)$ is called a {\it positively ordered monoid} if $\rW$ has an identity element $0$ and $0\preceq a$ for all $a\in \rW$ \cite[page 172]{APT18}.

The following nice result on extension of states for positively ordered monoid is a special case of \cite[Lemma 5.2.3]{APT18}, which applies more generally for positively ordered monoids not necessarily with order-units.

\begin{lemma} \label{L-APT}
Let $(\rW, \preceq, v)$ be a positively ordered monoid with order-unit. Let $\rW_1$ be a submonoid of $\rW$ containing $v$, and let $\varphi$ be a state of
$(\rW_1, \preceq, v)$. Let $a\in \rW$. Set
$$p=\sup\{(\varphi(b)-\varphi(c))/m \mid  b, c\in \rW_1, m\in \Nb, b\preceq c+ma\},$$ and
$$q=\inf\{(\varphi(b)-\varphi(c))/m \mid b, c\in \rW_1, m\in \Nb, b\succeq c+ma\}.$$
Then $0\le p\le q< \infty$ and
$$[p, q]=\{\psi(a) \mid \psi \mbox{ is a state of } (\rW_1+\Zb_{\ge 0} a, \preceq, v) \mbox{ extending } \varphi\}.$$
\end{lemma}

The following analogue of Theorem~\ref{T-GH} for positively ordered monoids with order-unit, due to Antoine, Perera, and Thiel, is a direct consequence of Lemma~\ref{L-APT} and Zorn's lemma.
The last sentence was proven earlier by Fuchssteiner \cite[Theorem 1]{Fuchssteiner} and Blackadar and R{\o}rdam \cite[Corollary 2.7]{BR92}.

\begin{theorem} \label{T-GH1}
Let $(\rW, \preceq, v)$ be a positively ordered monoid with order-unit.
Let $\rW_1$ be a submonoid of $\rW$ containing $v$, and let $\varphi$ be a state of $(\rW_1, \preceq, v)$. Let $a\in \rW$. Set $$p=\sup\{(\varphi(b)-\varphi(c))/m \mid  b, c\in \rW_1, m\in \Nb, b\preceq c+ma\},$$ and
$$q=\inf\{(\varphi(b)-\varphi(c))/m \mid b, c\in \rW_1, m\in \Nb, b\succeq c+ma\}.$$ Then $0\le p\le q<\infty$, and
$$[p, q]=\{\psi(a) \mid  \psi \mbox{ is a state of } (\rW, \preceq, v) \mbox{ extending } \varphi\}.$$ In particular,  every state of $(\rW_1, \preceq, v)$ extends to a state of $(\rW, \preceq, v)$.
\end{theorem}

%%%%%%%%%%%%%%%%%%%%%%%%%%%%%%%%%%%%%%%%%%%%%%%%%%%%%%%%%%%%%%%%%%%%%%%%%%%%%%%%
\subsection{Cuntz semigroup and dimension function} \label{SS-Cuntz}

We refer the reader to \cite{Cuntz78, BH, APT} for detail on Cuntz semigroup and dimension function.

Let $\cA$ be a $C^*$-algebra. Denote by $M(\cA)$ the space of all rectangular matrices over $\cA$. For $A, B\in M(\cA)$, we say $A$ is {\it Cuntz subequivalent} to $B$, written as $A\preceq_\rC B$, if there are sequences $\{C_n\}_{n\in \Nb}$ and $\{D_n\}_{n\in \Nb}$ in $M(\cA)$ such that $C_nBD_n\to A$ in norm as $n\to \infty$. This is a reflexive and transitive relation on $M(\cA)$, and if $A\preceq_\rC B$ and $C\preceq_\rC D$, then
$\diag(A, C)\preceq_\rC \diag(B, D)$.
 If $A\preceq_\rC B$ and $B\preceq_\rC A$, we say that $A$ and $B$ are {\it Cuntz equivalent} and write $A\sim_\rC B$. Then $\sim_\rC$ is an equivalence relation on $M(\cA)$, $\diag(A, B)\sim_\rC \diag(B, A)$ for all $A, B\in M(\cA)$, and if $A\sim_\rC B$ and $C\sim_\rC D$, then
$\diag(A, C)\sim_\rC \diag(B, D)$. The {\it Cuntz semigroup} of $\cA$ is $\rW_\rC(\cA):=M(\cA)/\sim_\rC$. For each $A\in M(\cA)$, denote by $\langle A\rangle_\rC$ the equivalence class of $A$ in $\rW_\rC(\cA)$. For any $\langle A\rangle_\rC, \langle B\rangle_\rC\in \rW_\rC(\cA)$, $\langle A \rangle_\rC+\langle B\rangle_\rC:=\langle \diag(A, B)\rangle_\rC$ does not depend on the choice of the representatives $A$ and $B$. Thus $\rW_\rC(\cA)$ is an abelian monoid with identity $\langle 0\rangle_\rC$ and partial order given by $\langle A\rangle_\rC \preceq_\rC \langle B \rangle_\rC$ if $A\preceq_\rC B$.

\begin{proposition} \label{P-Cuntz}
The following hold:
\begin{enumerate}
\item $A\sim_\rC |A|\sim_\rC A^*$ for every $A\in M(\cA)$.
\item $A\sim_\rC A^{\theta}$ for any $A\in M(\cA)_{\ge 0}$ and $\theta>0$.
\item $B\preceq_\rC A$ for any $B, A\in M(\cA)$ with $0\le B\le A$.
\item If $\cA$ is the $C^*$-algebra $C(X)$ of continuous $\Cb$-valued functions on some compact Hausdorff space $X$, then for any $f, g\in \cA$, one has $f\preceq_\rC g$ iff $\supp(f)\subseteq \supp(g)$, where $\supp(f)=\{x\in X: f(x)\neq 0\}$ is the open support of $f$.
\end{enumerate}
\end{proposition}

As indicated by Proposition~\ref{P-Cuntz}.(1), $\rW_\rC(\cA)$ can be defined equivalently using only positive matrices over $\cA$, which is how $\rW_\rC(\cA)$ is defined in most literature. Though we won't use it in this work, let us point out that another semigroup ${\rm Cu}(\cA)=\rW_\rC(\Kb \otimes \cA)$, where $\Kb\otimes \cA$ is the stablization of $\cA$, was introduced in \cite{CEI} and is heavily studied.

Now assume that $\cA$ is unital.  Then $\langle 1\rangle_{\rC}$ is an order-unit of the partially ordered abelian semigroup  $(\rW_\rC(\cA), \preceq_{\rC})$. A {\it dimension function} for $\cA$ is a function $\varphi: M(\cA)\rightarrow \Rb_{\ge 0}$ satisfying that  $\varphi(A)\le \varphi(B)$ whenever $A\preceq_\rC B$, $\varphi(1)=1$, and $\varphi(\diag(A, B))=\varphi(A)+\varphi(B)$. Clearly there is a natural one-to-one correspondence between dimension functions $\varphi$ of $\cA$ and states $\psi$ of $(\rW_\rC(\cA), \preceq_{\rC}, \langle 1\rangle_{\rC})$: every $\varphi$ is of the form $\psi\circ \pi$ for a unique $\psi$, where $\pi$ is the quotient map $M(\cA)\rightarrow \rW_\rC(\cA)$.

%%%%%%%%%%%%%%%%%%%%%%%%%%%%%%%%%%%%%%%%%%%%%%%%%%%%%%%%%%%%%%%%%%%%%%%%%%%%%%%%%%
\subsection{Sylvester rank function} \label{SS-rank function}

For general references on Sylvester matrix and module rank functions see \cite{Malcolmson80, Schofield, JZ17}.
Let $R$ be a unital ring. Denote by $M(R)$ the set of all rectangular matrices over $R$.

\begin{definition} \label{D-Sylvester matrix}
A {\it Sylvester matrix rank function} for $R$ is a function $\rk: M(R)\rightarrow \Rb_{\ge 0}$ satisfying the following conditions:
\begin{enumerate}
\item $\rk(0)=0$ and $\rk(1)=1$.
\item $\rk(AB)\le \min(\rk(A), \rk(B))$.
\item $\rk(\left[\begin{matrix} A & \\ & B \end{matrix}\right])=\rk(A)+\rk(B)$.
\item $\rk(\left[\begin{matrix} A & C\\ & B \end{matrix}\right])\ge \rk(A)+\rk(B)$.
\end{enumerate}
\end{definition}

We shall denote by $\Pb(R)$ the set of all Sylvester matrix rank functions for $R$. It is naturally a compact convex space under the topology of pointwise convergence. When $R$ is a division ring, it has a unique Sylvester matrix rank function, i.e. $\rk(A)$ is the dimension of the row space or column space of $A$ over $R$.

Denote by $\sFPM(R)$ the class of finitely presented left $R$-modules.

\begin{definition} \label{D-Sylvester mod}
A {\it Sylvester module rank function} for $R$ is a function $\dim: \sFPM(R)\rightarrow \Rb_{\ge 0}$ satisfying the following conditions:
\begin{enumerate}
\item $\dim(0)=0$ and $\dim(R)=1$.
\item $\dim(\cM_1\oplus \cM_2)=\dim(\cM_1)+\dim(\cM_2)$.
\item For any exact sequence $\cM_1\rightarrow \cM_2\rightarrow \cM_3\rightarrow 0$, one has
$$\dim(\cM_3)\le \dim(\cM_2)\le \dim(\cM_1)+\dim(\cM_3).$$
\end{enumerate}
\end{definition}

From (3) it is clear that $\dim$ is an isomorphism invariant.
The following is \cite[Theorem 4]{Malcolmson80}.

\begin{theorem} \label{T-bijection}
There is a natural one-to-one correspondence between Sylvester matrix rank functions and Sylvester module rank functions for $R$ as follows:
\begin{enumerate}
\item Given a Sylvester matrix rank function $\rk$ for $R$, for any $\cM\in \sFPM(R)$ take some $A\in M_{n, m}(R)$ such that $\cM\cong R^m/R^nA$, and define $\dim(\cM)=m-\rk(A)$. Then $\dim$ is a Sylvester module rank function for $R$.
\item Given a Sylvester module rank function $\dim$ for $R$, for any $A\in M_{n, m}(R)$, define $\rk(A)=m-\dim(R^m/R^nA)$. Then $\rk$ is a Sylvester matrix rank function for $R$.
\end{enumerate}
\end{theorem}

A unital ring $R$ is {\it von Neumann regular} \cite[page 61]{Lam01} \cite[page 1]{Goodearl} if for any $a\in R$ one has $aba=a$ for some $b\in R$.

\begin{definition} \label{D-regular rank}
A Sylvester matrix rank function $\rk$ for $R$ is {\it regular} if there are a unital ring homomorphism $\varphi$ from $R$ to some von Neumann regular ring $S$ and a $\rk_S\in \Pb(S)$  such that $\rk(A)=\rk_S(\varphi(A))$ for all $A\in M(R)$.
\end{definition}

We shall denote by $\Pb_{\rm reg}(R)$ the set of all regular Sylvester matrix rank functions for $R$. This is a closed convex subset of $\Pb(R)$ \cite[Proposition 5.9]{JZ19}.

%%%%%%%%%%%%%%%%%%%%%%%%%%%%%%%%%%%%%%%%%%%%%%%%%%%%%%%%%%%%%%%%%%%%%%%%%%%%%%%%%%%%%%%%%%%%%%%%%%%%%%%%%%%%%%%%%%%%%%%%%%%%%%%%%%%%%%%%%%%%%%%%%%%
\section{Matrix Malcolmson semigroup} \label{S-Matrix M}

Throughout this section let $R$ be a unital ring.
The following is inspired by the construction of the Cuntz semigroup.

\begin{definition} \label{D-Matrix Malcolmson}
For $A, B\in M(R)$, we write $A\lesssim B$ if either $A=CBD$ for some $C, D\in M(R)$ or $B=\left[\begin{matrix} C & E\\ & D \end{matrix}\right]$ and $A=\left[\begin{matrix} C & \\ & D \end{matrix}\right]$ for some $C, D, E\in M(R)$. We say $A$ is {\it Malcolmson subequivalent} to $B$, written as $A\preceq_\rM B$, if there are $n\in \Nb$ and  $A_1, \dots, A_n\in M(R)$ such that $A=A_1\lesssim A_2\lesssim\dots \lesssim A_n=B$.
This is a reflexive and transitive relation on $M(R)$. If $A\preceq_\rM B$ and $B\preceq_\rM A$, we say that $A$ and $B$ are {\it Malcolmson equivalent} and write $A\sim_\rM B$. Then $\sim_\rM$ is an equivalence relation on $M(R)$. We define the {\it matrix Malcolmson semigroup} of $R$ as $\rW_\rM(R):=M(R)/\sim_\rM$.
For each $A\in M(R)$, denote by $\langle A\rangle_\rM$ the equivalence class of $A$ in $\rW_\rM(R)$.
\end{definition}

Note that if $A\preceq_\rM B$ and $C\preceq_\rM D$, then $\diag(A, C)\preceq_\rM \diag(B, D)$.  Also note that $\diag(A, B)\sim_\rM \diag(B, A)$ for all $A, B\in M(R)$, and if $A\sim_\rM B$ and $C\sim_\rM D$, then
$\diag(A, C)\sim_\rM \diag(B, D)$. For any $\langle A\rangle_\rM, \langle B\rangle_\rM\in \rW_\rM(R)$, $\langle A \rangle_\rM+\langle B\rangle_\rM:=\langle \diag(A, B)\rangle_\rM$ does not depend on the choice of the representatives $A$ and $B$. Thus $\rW_\rM(R)$ is an abelian monoid with identity $\langle 0\rangle_\rM$ and partial order given by $\langle A\rangle_\rM \preceq_\rM \langle B \rangle_\rM$ if $A\preceq_\rM B$.
The element $\langle 1\rangle_\rM$ is an order-unit of $(\rW_\rM(R), \preceq_\rM)$.

Denote by $\pi$ the quotient map $M(R)\rightarrow \rW_\rM(R)$. Clearly the Sylvester matrix rank functions of $R$ are exactly $\varphi\circ \pi$ for states $\varphi$ of the partially ordered abelian semigroup with order-unit $(\rW_\rM(R), \preceq_\rM, \langle 1\rangle_\rM)$. From Corollary~\ref{C-range} we obtain the following.

\begin{proposition} \label{P-existence}
$R$ has a Sylvester rank function if and only if $I_{n+1}\not\preceq_{\rM} I_n$ for all $n\in \Nb$.
\end{proposition}

Recall that $R$ is said to satisfy the {\it rank condition} if for any $n\in \Nb$ the free left $R$-module $R^n$ is not generated by $n-1$ elements \cite[Definition 1.20]{Lam}.

\begin{proposition} \label{P-function to rank}
If $R$ has a Sylvester rank function, then $R$ satisfies the rank condition.
\end{proposition}
\begin{proof} Suppose that $R$ does not satisfy the rank condition. Then there are $n<m$ in $\Nb$ and $A\in M_{m, n}(R), B\in M_{n, m}(R)$ such that $AB=I_m$ \cite[Proposition 1.24]{Lam}. If $R$ has a Sylvester matrix rank function $\rk$, then
$$ m=\rk(I_m)=\rk(AB)=\rk(AI_nB)\le \rk(I_n)=n,$$
which is impossible.
\end{proof}

It will be interesting to know whether the rank condition characterises the existence of Sylvester rank functions.

\begin{question} \label{Q-rank to function}
If $R$ satisfies the rank condition, then must $R$ have a Sylvester rank function?
\end{question}

Question~\ref{Q-rank to function} has an affirmative answer for von Neumann regular rings \cite[Theorem 18.3]{Goodearl}.
Recall that a unital ring $R$ is said to be {\it stably finite} if for any $n\in \Nb$ and $A, B\in M_n(R)$ with $AB=I_n$, one has $BA=I_n$.
It is a result of Malcolmson  that $R$ satisfies the rank condition if and only if $R$ has a stably finite quotient ring \cite[Theorem 1.26]{Lam} \cite{Malcolmson79}.
Thus Question~\ref{Q-rank to function} is equivalent to the following question.

\begin{question} \label{Q-stable finite to function}
If $R$ is stably finite, then must $R$ have a Sylvester rank function?
\end{question}

We remark that a result of Cuntz \cite[Corollary 4.5]{Cuntz78} says that  stably finite unital $C^*$-algebras have dimension functions. In view of Proposition~\ref{P-M to C} below this implies an affirmative answer of Question~\ref{Q-stable finite to function} for unital $C^*$-algebras.

As an example, we compute the matrix Malcolmson semigroup for left Artinian local rings whose maximal ideal is generated by a central element. For this we need the following lemma.

\begin{lemma} \label{L-compare power1}
Let $A\in M_n(R)$. Let $0\le k\le j$ and $0<m$ be integers. Then $\diag(A^j, A^{k+m})\preceq_\rM \diag(A^k, A^{j+m})$, where we set $A^0=I_n$.
\end{lemma}
\begin{proof} For any $B, C, D\in M_l(R)$ with $B, D$ invertible, one has $BCD\sim_\rM C$. Thus using elementary row and column operations we have
\begin{align*}
\left[\begin{matrix} A^k & \\  & A^{j+m} \end{matrix}\right]&\sim_\rM \left[\begin{matrix} I_n & \\ A^m & I_n \end{matrix}\right]\cdot \left[\begin{matrix} A^k & \\  & A^{j+m} \end{matrix}\right]\cdot \left[\begin{matrix} I_n & A^{j-k}\\  & -I_n \end{matrix}\right]\cdot \left[\begin{matrix}  & I_n \\ I_n  &  \end{matrix}\right]\\
&=\left[\begin{matrix} I_n & \\ A^m & I_n \end{matrix}\right]\cdot \left[\begin{matrix} A^k & A^j\\  & -A^{j+m} \end{matrix}\right]\cdot \left[\begin{matrix}  & I_n \\ I_n  &  \end{matrix}\right]\\
&=\left[\begin{matrix} A^j & A^k \\   & A^{k+m} \end{matrix}\right].
\end{align*}
From the second case in the definition of $\lesssim$ one also has
$$ \left[\begin{matrix} A^j & A^k\\ & A^{k+m} \end{matrix}\right]\gtrsim \left[\begin{matrix} A^j & \\  & A^{k+m} \end{matrix}\right].$$
Now the assertion follows.
\end{proof}

Recall that a unital ring $R$ is {\it local} if the set of non-invertible elements of $R$ is an ideal \cite[Section 19]{Lam01}. In such case, this ideal is equal to the Jacobson radical $\rad(R)$ of $R$. Also recall that if a unital ring $R$ is {\it left Artinian}, i.e. there is no infinite strictly descending chain of left ideals $\cJ_1\supsetneq \cJ_2\supsetneq \cdots$ \cite[page 19]{Lam01}, then $\rad(R)$ is nilpotent \cite[Theorem 4.12]{Lam01}, i.e. $(\rad(R))^n=0$ for some $n\in \Nb$.

Let $R$ be a left Artinian local ring such that $\rad(R)$ is generated by some central element $c$. Let $n$ be the smallest positive integer such that $c^n=0$. Examples include $R=\Db[x]/x^n \Db[x]$, where $\Db$ is a division ring and $\Db[x]$ is the  polynomial ring over $\Db$, and $R=\Zb/p^n\Zb$, where $p$ is a prime number.

By \cite[Corollary 2.4]{JL22} for each $1\le k\le n$ there is a $\rk_k\in \Pb(R)$ satisfying
\begin{align} \label{E-extreme}
\rk_k(c^i)=\begin{cases}
\frac{k-i}{k} &\text{ if $0\le i< k$}\\
0 &\text{ if $k\le i\le n-1$}.
\end{cases}
\end{align}
It is further shown in \cite[Corollary 2.4]{JL22} that $\rk_1, \dots, \rk_n$ are exactly the extreme points of the convex set $\Pb(R)$, which we shall need later in Example~\ref{E-nilpotent}. Below we shall identify $\rk\in \Pb(R)$ with the corresponding state of $(\rW_\rM(R), \preceq_\rM, \langle 1\rangle_\rM)$.

\begin{proposition} \label{P-local}
Let $R$ be a left Artinian unital local ring such that $\rad(R)$ is generated by a central element $c$. Let $n$ be the smallest positive integer such that $c^n=0$. Then $\rW_\rM(R)$ is the free abelian monoid generated by $\langle c^i\rangle_\rM$ for $i=0, 1, \dots, n-1$, i.e. $\rW_\rM(R)=\bigoplus_{0\le i\le n-1}\Zb_{\ge 0}\langle c^i\rangle_\rM$. Furthermore, for any $a, b\in \rW_\rM(R)$, the following are equivalent:
\begin{enumerate}
\item $a\preceq_\rM b$;
\item $\rk(a)\le \rk(b)$ for all $\rk\in \Pb(R)$;
\item $\rk_k(a)\le \rk_k(b)$ for all $k=1, \dots, n$.
\end{enumerate}
\end{proposition}
\begin{proof}
Note that every nonzero element of $R$ is of the form $xc^i$ for some invertible $x\in R$ and $0\le i\le n-1$.
We claim that for any $A\in M_m(R)$, there are some invertible $B, C\in M_m(R)$ such that $BAC=\diag(c^{i_1}, \dots, c^{i_l}, 0, \dots, 0)$ for some $0\le l\le m$ and $0\le i_1, \dots, i_l\le n-1$. We argue by induction on $m$. The case $m=1$ is obvious. Assume that the claim holds for matrices in $M_{m-1}(R)$. Let $A\in M_m(R)$. The claim is trivial if $A$ is the zero matrix. Thus we may assume that $A$ is nonzero. Writing each nonzero entry of $A$ as $xc^i$ for some invertible $x\in R$ and $0\le i\le n-1$, we can find one nonzero entry of $A$ with the smallest exponent $i$. Via exchanging some rows and columns of $A$, which be can achieved via replacing $A$ by $BAC$ for some invertible $B, C\in M_m(R)$, we may assume that this entry is at the upper-left corner of $A$. Via elementary row and column operations, which can be achieved again via  replacing $A$ by $BAC$ for some invertible $B, C\in M_m(R)$, we may use this entry to eliminate all other entries in the first row and column of $A$ so that $A$ becomes $\diag(xc^i, A')$ for some  invertible $x\in R$, some $0\le i\le n-1$, and some $A'\in M_{m-1}(R)$. By induction hypothesis we have $B'A'C'=\diag(c^{i_1}, \dots, c^{i_l}, 0, \dots, 0)$ for some invertible $B', C'\in M_{m-1}(R)$, some $0\le l\le m-1$ and $0\le i_1, \dots, i_l\le n-1$. Then $B=\diag(x^{-1}, B')$ and $C=\diag(1, C')$ are invertible in $M_m(R)$, and $BAC=\diag(c^i, c^{i_1}, \dots, c^{i_l}, 0, \dots, 0)$. This finishes the induction step and proves our claim.

From the above claim it follows that $\rW_\rM(R)$ is generated by $\langle c^i\rangle_\rM$ for $i=0, 1, \dots, n-1$.

For each $a\in \rW_\rM(R)$, write $a$ as $\sum_{i=0}^{n-1}a_i\langle c^i\rangle_\rM$ with $a_0, \dots, a_{n-1}\in \Zb_{\ge 0}$. If $a=b$, then $$\sum_{i=0}^{k-1}(k-i)a_i=k\cdot \rk_k(a)=k\cdot \rk_k(b)=\sum_{i=0}^{k-1}(k-i)b_i$$ for every $1\le k\le n$, from which it follows that $a_i=b_i$ for all $0\le i\le n-1$. Thus $\rW_\rM(R)$ is the free abelian monoid generated by $\langle c^i\rangle_\rM$ for $i=0, 1, \dots, n-1$. For each $a\in \rW_\rM(R)$, denote by $\supp(a)$ the set of $0\le i\le n-1$ with $a_i\neq 0$, and set $\|a\|_1=\sum_{i=0}^{n-1}a_i$.

(1)$\Rightarrow$(2)$\Rightarrow$(3) are trivial. To prove (3)$\Rightarrow$(1), let $a, b\in \rW_\rM(R)$ such that
$\rk_k(a)\le \rk_k(b)$ for all $k=1, \dots, n$. We shall show $a\preceq_\rM b$ by induction on $\|a\|_1$. The case $\|a\|_1=0$ is trivial.
Assume that the claim holds when $\|a\|_1=l$. Consider the case $\|a\|_1=l+1$.

If there is some $i\in \supp(a)\cap \supp(b)$, then $a-\langle c^i\rangle_\rM, b-\langle c^i\rangle_\rM\in \rW_\rM(R)$,  $\rk_k(a-\langle c^i\rangle_\rM)\le \rk_k(b-\langle c^i\rangle_\rM)$ for all $1\le k\le n$, and $\|a-\langle c^i\rangle_\rM\|_1=l$, thus by the induction hypothesis we have $a-\langle c^i\rangle_\rM\preceq_\rM b-\langle c^i\rangle_\rM$, whence $a\preceq_\rM b$. Therefore we may assume that $\supp(a)\cap \supp(b)=\emptyset$.

Put $i_a=\min(\supp(a))$ and $i_b=\min(\supp(b))$. If $i_a<i_b$, then $\rk_{i_a+1}(a)>0=\rk_{i_a+1}(b)$, which is impossible. Thus $i_a>i_b$. Denote by $j_1$ the largest element of $\supp(b)$ smaller than $i_a$.
If $j_1\neq \max(\supp(b))$, let $j_2$ be the smallest element in $\supp(b)$ greater than $j_1$. Otherwise, put $j_2=n$.
Then $i_b\le j_1<i_a<j_2$.
By Lemma~\ref{L-compare power1} we have
$$\langle c^{j_1}\rangle_\rM+\langle c^{j_2}\rangle_\rM\succeq_\rM \langle c^{j_1+1}\rangle_\rM+\langle c^{j_2-1}\rangle_\rM.$$ Put
$$b^{(1)}=b-\langle c^{j_1}\rangle_\rM-\langle c^{j_2}\rangle_\rM+\langle c^{j_1+1}\rangle_\rM+\langle c^{j_2-1}\rangle_\rM\in \rW_\rM(R).$$
Then $b^{(1)}\preceq_\rM b$.

We claim that $\rk_k(a)\le \rk_k(b^{(1)})$ for all $1\le k\le n$. Assume that $\rk_k(a)>\rk_k(b^{(1)})$ for some $1\le k\le n$ instead. Then $\rk_k(b)\ge \rk_k(a)>\rk_k(b^{(1)})$, whence $j_1<k<j_2$. Note that $k\cdot \rk_k(b), k\cdot \rk_k(a), k\cdot \rk_k(b^{(1)})$ are all integers, and $k\cdot \rk_k(b)=k\cdot \rk_k(b^{(1)})+1$. Thus
$$k\cdot \rk_k(b)=k\cdot \rk_k(a)=k\cdot \rk_k(b^{(1)})+1.$$  We have
$$ k\cdot \rk_k(b)=\sum_{i\in \supp(b), i<k}(k-i)b_i\ge (k-j_1)\sum_{i\in \supp(b), i<k}b_i>0,$$
and
$$k\cdot \rk_k(a)=\sum_{i\in \supp(a), i<k}(k-i)a_i\le (k-i_a)\sum_{i\in \supp(a), i<k}a_i.$$
Since $k-i_a<k-j_1$, we obtain $0<k-i_a$ and
$$\sum_{i\in \supp(b), i<k}b_i<\sum_{i\in \supp(a), i<k}a_i.$$
Note that
$$(k+1)\cdot \rk_{k+1}(b)=k\cdot \rk_k(b)+\sum_{i\in \supp(b), i\le k}b_i=k\cdot \rk_k(b)+\sum_{i\in \supp(b), i<k}b_i,$$
whence
\begin{align*}
(k+1)\cdot \rk_{k+1}(a)&=k\cdot \rk_k(a)+\sum_{i\in \supp(a), i\le k}a_i\ge k\cdot \rk_k(a)+\sum_{i\in \supp(a), i<k}a_i\\
&>k\cdot \rk_k(b)+\sum_{i\in \supp(b), i<k}b_i=(k+1)\cdot \rk_{k+1}(b),
\end{align*}
which is impossible. This proves our claim.

Iterating the above construction, we find $b^{(1)}, \dots, b^{(m)}$ in $\rW_\rM(R)$ such that $b\succeq_\rM b^{(1)}\succeq_\rM \dots \succeq_\rM b^{(m)}$, and $\rk_k(a)\le \rk_k(b^{(m)})$ for all $1\le k\le n$, and $\supp(a)\cap \supp(b^{(m)})\neq \emptyset$. As at the beginning of the induction step, we conclude $b^{(m)}\succeq_\rM a$. Therefore $b\succeq_\rM a$. This finishes the induction step and proves (3)$\Rightarrow$(1).
\end{proof}

%%%%%%%%%%%%%%%%%%%%%%%%%%%%%%%%%%%%%%%%%%%%%%%%%%%%%%%%%%%%%%%%%%%%%%%%%%%%%%%%%%%%%%%%%%%%%%%%%%%%%%%%%%%%%%%%%%%%%%%%%%%%%%%%%%%%%%%%%%%%%%%%%
\section{Sylvester rank functions for commutative rings} \label{S-commutative}

In this section we use the matrix Malcolmson semigroup to construct some interesting Sylvester matrix rank functions for commutative rings in Theorem~\ref{T-rk for square}. Throughout this section, $R$ will be a unital commutative ring.

For a unital ring $R$, recall that $a\in R$ is {\it nilpotent} if $a^n=0$ for some $n\in \Nb$.

\begin{proposition} \label{P-rank for commutative}
The following hold:
\begin{enumerate}
\item $\Pb_{\rm reg}(R)\neq \emptyset$.
\item If $a\in R$ is invertible, then
$$\{\rk(a): \rk\in \Pb_{\rm reg}(R)\}=\{\rk(a): \rk\in \Pb(R)\}=\{1\}.$$
\item If $a\in R$ is neither invertible nor nilpotent, then
$$\{\rk(a): \rk\in \Pb_{\rm reg}(R)\}=[0, 1].$$
\item If $a\in R$ with $a^m=0$ for some $m\in \Nb$, then 
$$\{0\}\subseteq \{\rk(a): \rk\in \Pb_{\rm reg}(R)\}\subseteq \{\rk(a): \rk\in \Pb(R)\}\subseteq [0, (m-1)/m].$$
\end{enumerate}
\end{proposition}
\begin{proof}
(1). By Zorn's lemma $R$ has a maximal ideal $\cJ$. Then $R/\cJ$ is a field, and hence $R/\cJ$ has a unique Sylvester matrix rank function. It pulls back to a regular Sylvester matrix rank function for $R$.

(2) follows from (1).

(3). Let $a\in R$ such that $a$ is not invertible. Then by Zorn's lemma $a$ is contained in some maximal ideal $\cJ$ of $R$. As in the proof of (1), $R$ has a unique regular Sylvester matrix rank function vanishing on $\cJ$. Thus $0\in \{\rk(a): \rk\in \Pb_{\rm reg}(R)\}$.

Let $a\in R$ such that $a$ is not nilpotent.
By Zorn's lemma $R$ has an ideal $\cJ$ being maximal among the ideals which do not intersect with $\{a^n: a\in \Nb\}$. It is easily checked that $\cJ$ is a prime ideal of $R$. Then the fraction field of $R/\cJ$ has  a unique Sylvester matrix rank function. It pulls back to a regular Sylvester matrix rank function $\rk$ for $R$. Clearly $\rk(a)=1$. Thus $1\in \{\rk(a): \rk\in \Pb_{\rm reg}(R)\}$.

The assertion follows from the above two paragraphs and the fact that $\Pb_{\rm reg}(R)$ is a convex set.

(4). Since $a$ is not invertible, the proof of (3) shows that $\{0\}\subseteq \{\rk(a): \rk\in \Pb_{\rm reg}(R)\}$.

When $m=1$ one has $a=0$, whence $\rk(a)=0$ for every $\rk\in \Pb(R)$. Thus we may assume that $m\ge 2$. 
From Lemma~\ref{L-compare power1} we have $\langle a^j\rangle_{\rM}+\langle a\rangle_\rM \preceq_\rM \langle 1\rangle_\rM+\langle a^{j+1}\rangle_\rM$ for every $j\in \Zb_{\ge 0}$. Then
\begin{align*}
 m\langle a\rangle_\rM &\preceq_\rM \langle 1\rangle_\rM+\langle a^2\rangle_\rM+(m-2)\langle a\rangle_\rM \\
 &\preceq_\rM 2\langle 1\rangle_\rM+\langle a^3\rangle_\rM+(m-3)\langle a\rangle_\rM \\
 &\preceq_\rM\dots \preceq_\rM (m-1) \langle 1\rangle_\rM+\langle a^m \rangle_\rM=(m-1)\langle 1\rangle_\rM. 
 \end{align*}
Thus for every $\rk\in \Pb(R)$ one has $m\cdot \rk(a)\le m-1$. 
\end{proof}

\begin{example} \label{E-nilpotent}
Let $R=\Kb[x]/x^n \Kb[x]$ for some field $\Kb$ and some even integer $n>6$. Put $c=x+x^n\Kb[x]\in R$ and $a=c^{\frac{n}{2}-1}\in R$. Then $a^3=0$ and $a^2\neq 0$.
By \cite[Corollary 2.4]{JL22} the extreme points of $\Pb(R)$ are $\rk_1, \rk_2, \dots, \rk_n$ satisfying \eqref{E-extreme}. Thus
$$\max_{\rk \in \Pb(R)} \rk(a)=\max_{1\le k\le n}\rk_k(a)=\rk_n(a)=\frac{n+2}{2n}<\frac{2}{3}.$$
This shows that the inclusion $\{\rk(a): \rk\in \Pb(R)\}\subseteq [0, (m-1)/m]$ in Proposition~\ref{P-rank for commutative}.(4) may be proper even when $m$ is the smallest $j\in \Nb$ with $a^j=0$.
\end{example}

Recall that for any $A\in M_{m, n}(R)$ and any integer $0<k\le \min(m, n)$, a {\it $k\times k$-minor} is the determinant of a $k\times k$-submatrix of $A$. As a convention, for any integer $k>\min(m, n)$, we say that $0$ is  the unique $k\times k$-minor of $A$.

\begin{lemma} \label{L-minor}
Let $A, B\in M(R)$ such that $A\preceq_\rM B$. Let $\cJ$ be an ideal of $R$ and $k\in \Nb$ such that every $k\times k$-minor of $B$ lies in $\cJ$. Then every $k\times k$-minor of $A$ lies in $\cJ$.
\end{lemma}
\begin{proof} It suffices to consider three cases: $A=BC$ for some $C\in M(R)$, $A=CB$ for some $C\in M(R)$, and $B=\left[\begin{matrix} C & E\\ & D \end{matrix}\right]$ and $A=\left[\begin{matrix} C & \\ & D \end{matrix}\right]$ for some $C, D, E\in M(R)$.

Consider first the case $A=BC$ for some $C\in M(R)$. Say, $B\in M_{m, n}(R)$ and $C\in M_{n, l}(R)$. We may assume that $k\le \min(m, l)$. Let $A'$ be a $k\times k$-submatrix of $A$. Then $A'=B'C'$ for some $k\times n$-submatrix $B'$ of $B$ and some $n\times k$-submatrix $C'$ of $C$. Thus each column of $A'$ is a linear combination of columns of $B'$ with coefficients being entries of $C'$. It follows that $\det(A')$ is a finite sum of elements of the form $\pm c\det(A'')$ such that $c$ is a product of entries in $C'$ and $A''\in M_{k}(R)$ with each column of $A''$ being a column of $B'$. If $A''$ has two identical columns, then $\det(A'')=0$. Otherwise, $A''$ is a $k\times k$-submatrix of $B$ with the columns permuted, whence $\det(A'')\in \cJ$. Therefore, $\det(A')\in \cJ$.

The case $A=CB$ for some $C\in M(R)$ is handled similarly, via considering rows instead of columns.

Finally consider the case $B=\left[\begin{matrix} C & E\\ & D \end{matrix}\right]$ and $A=\left[\begin{matrix} C & \\ & D \end{matrix}\right]$ for some $C, D, E\in M(R)$. Again, we may assume that $A$ has at least $k$ rows and at least $k$ columns. Let $A'$ be a $k\times k$-submatrix of $A$. Then $A'=\left[\begin{matrix} C' & \\ & D' \end{matrix}\right]$ for some submatrix $C'$ of $C$ and some submatrix $D'$ of $D$. Note that there is a submatrix $E'$ of $E$ such that $B'=\left[\begin{matrix} C' & E'\\ & D' \end{matrix}\right]$ is a $k\times k$-submatrix of $B$. If one of $C'$ and $D'$ is not a square matrix, then $\det(A')=0\in \cJ$. Thus we may assume that both $C'$ and $D'$ are square matrices. Then $\det(A')=\det(C')\det(D')=\det(B')\in \cJ$.
\end{proof}

\begin{lemma} \label{L-upper bound}
Let $a\in R$, and  let $m, n, j, k, l\ge 0$. Assume that $a^{2|n-k|+1}\not\in Ra^{2|n-k|+2}$. Also assume that $k\langle 1\rangle_\rM+j\langle a^2\rangle_\rM +m\langle a\rangle_\rM\preceq_\rM n\langle 1\rangle_\rM+l\langle a^2\rangle_\rM$. Then $m\le 2(n-k)$.
\end{lemma}
\begin{proof} Denote by $A$ an $(k+j+m)\times (k+j+m)$-diagonal matrix with $k$ diagonal entries being $1$, $j$ diagonal entries being $a^2$, and  $m$ diagonal entries being $a$. Also
denote by $B$ an $(n+l)\times (n+l)$-diagonal matrix with $n$ diagonal entries being $1$ and $l$ diagonal entries being $a^2$.
Then $A\preceq_\rM B$.

Since $a^{2|n-k|+1}\not\in Ra^{2|n-k|+2}$, we have $1\not\in Ra$.

We claim that $n\ge k$. Assume that $n<k$ instead. Note that every $(n+1)\times (n+1)$-minor of $B$ lies in the ideal $Ra$. By Lemma~\ref{L-minor} every $(n+1)\times (n+1)$-minor of $A$ lies in $Ra$. But clearly $1$ is an $(n+1)\times (n+1)$-minor of $A$, a contradiction. This proves our claim.

Put $K=2n-k+1$. Note that every $K\times K$-minor of $B$ lies in the ideal $Ra^{2(n-k+1)}$. By Lemma~\ref{L-minor} every $K\times K$-minor of $A$ lies in $Ra^{2(n-k+1)}$. Assume that $m> 2n-2k$ instead. Then $m\ge 2n-2k+1$. Clearly $a^{2n-2k+1}$ is a $K\times K$-minor of $A$. This contradicts our assumption that $a^{2|n-k|+1}\not\in Ra^{2|n-k|+2}$. Therefore $m\le 2n-2k$.
\end{proof}

\begin{theorem} \label{T-rk for square}
Let $a\in R$ such that $a^n\not\in Ra^{n+1}$ for all $n\ge 0$, where we put $a^0=1$. Then
$$\{\rk(a): \rk\in \Pb(R), \rk(a^2)=0\}=[0, 1/2].$$
\end{theorem}
\begin{proof} Taking $k=0$ and $j=m=1$ in Lemma~\ref{L-compare power1} we have $\diag(a, a)\preceq_\rM\diag(1, a^2)$, whence $2\rk(a)\le 1+\rk(a^2)$ for every $\rk\in \Pb(R)$. Therefore
$$\{\rk(a): \rk\in \Pb(R), \rk(a^2)=0\}\subseteq [0, 1/2].$$

Since $\Pb(R)$ is a convex set, to show
$$\{\rk(a): \rk\in \Pb(R), \rk(a^2)=0\}\supseteq [0, 1/2],$$ it suffices to show that there are $\rk, \rk'\in \Pb(R)$ such that $\rk(a^2)=\rk'(a^2)=0$, and $\rk(a)=1/2$, $\rk'(a)=0$.

Since $1\not\in Ra$, the commutative quotient ring $R/aR$ is nonzero, thus by Proposition~\ref{P-rank for commutative}.(1) $R/aR$ has Sylvester matrix rank functions. Take any such rank function and pull it back to $R$, we obtain an $\rk'\in \Pb(R)$ such that $\rk'(a^2)=\rk'(a)=0$.

Let $\rW_1$ be the subsemigroup of $\rW_{\rM}(R)$ generated by $\langle 1\rangle_\rM$ and $\langle a^2\rangle_\rM$. Let $\rk_1$ be the restriction of $\rk'$ to $(\rW_1, \preceq_\rM, \langle 1\rangle_\rM)$. Put
$$\lambda:=\inf\{(\rk_1(b)-\rk_1(c))/m \mid b, c\in \rW_1, m\in \Nb, b\succeq_\rM c+m\langle a\rangle_\rM\}.$$
By Lemma~\ref{L-upper bound} we have $\lambda\ge 1/2$. From the first paragraph of the proof we have $\lambda\le 1/2$. Therefore $\lambda=1/2$.
By Theorem~\ref{T-extension submonoid}  there is a state $\rk$ of $(\rW_\rM(R), \preceq_\rM, \langle 1\rangle_\rM)$ extending $\rk_1$
such that $\rk(a)=\lambda=1/2$.
\end{proof}

Theorem~\ref{T-rk for square} will be used later in Example~\ref{E-rk not dim}.

\begin{remark} \label{R-regular}
It was shown in \cite[Example 2.1.13]{LA} that the rank function $\rk\in \Pb(\Zb/4\Zb)$ satisfying $\rk(2+4\Zb)=1/2$ is not regular. We don't know whether
Theorem~\ref{T-rk for square} still holds if we replace $\Pb(R)$ by $\Pb_{\rm reg}(R)$.
\end{remark}

To end this section, we record the following two simple lemmas, which will be used in Section~\ref{S-M vs C}.

\begin{lemma} \label{L-commutative}
Let $a, b\in R$. Then $\langle a\rangle_\rM\preceq_\rM \langle b\rangle_\rM$ if and only if $a\in Rb$.
\end{lemma}
\begin{proof} Taking $\cJ=Rb$ and $k=1$ in Lemma~\ref{L-minor} we get the ``only if'' part. If $a\in Rb$, say $a=xb$ for some $x\in R$, then $a\lesssim b$, whence $\langle a\rangle_\rM\preceq_\rM \langle b\rangle_\rM$, yielding the ``if'' part.
\end{proof}

\begin{lemma} \label{L-not nilpotent}
Let $a, b\in R$ such that $a^n\not\in Rb$ for all $n\in \Nb$. Then there is a $\rk\in \Pb_{\rm reg}(R)$ such that $\rk(a)=1$ and $\rk(b)=0$.
\end{lemma}
\begin{proof} Denote by $\pi$ the quotient map $R\rightarrow R':=R/Rb$. The assumption says that $\pi(a)$ is not nilpotent. The proof of Proposition~\ref{P-rank for commutative}.(3) shows that there is some $\rk'\in \Pb_{\rm reg}(R')$ such that $\rk'(\pi(a))=1$.
Then $\rk:=\rk'\circ \pi$ is in $\Pb_{\rm reg}(R)$. Clearly $\rk(a)=1$ and $\rk(b)=0$.
\end{proof}

%%%%%%%%%%%%%%%%%%%%%%%%%%%%%%%%%%%%%%%%%%%%%%%%%%%%%%%%%%%%%%%%%%%%%%%%%%%%%%%%%%%%%%%%%%%%%%%%%%%%%%%%%%%%%%%%%%%%%%%%%%
\section{Malcolmson semigroups and Sylvester rank functions for $C^*$-algebras} \label{S-M vs C}

In this section we compare the matrix Malcolmson semigroup (resp. Sylvester matrix rank functions) with the Cuntz semigroup (resp. dimension functions) of a unital $C^*$-algebra. Throughout this section $\cA$ will be  a unital $C^*$-algebra.

\begin{lemma} \label{L-M to C}
For any $A, B\in M(\cA)$, if $A\preceq_\rM B$, then $A\preceq_\rC B$.
\end{lemma}
\begin{proof} It suffices to show that if $A\lesssim B$, then $A\preceq_\rC B$. In turn it suffices to show that for any
$A\in M_{n, m}(\cA), B\in M_{k, l}(\cA)$ and $C\in M_{n, l}(\cA)$, one has $\left[\begin{matrix} A &  \\ & B \end{matrix}\right]\preceq_\rC \left[\begin{matrix} A & C \\ & B \end{matrix}\right]$.

 For any continuous $\Cb$-valued function $f$ on the spectrum $\spec(|A|)$ of $|A|$, we have
\begin{align*}
\left[\begin{matrix} |A|A^* & \\ & I_k \end{matrix}\right]\left[\begin{matrix} A & C\\ & B \end{matrix}\right]\left[\begin{matrix} I_m & -f(|A|)A^*C \\ & I_l \end{matrix}\right]=\left[\begin{matrix} |A|^3 & (|A|-|A|^3f(|A|))A^*C\\ & B \end{matrix}\right].
\end{align*}
Take a sequence $\{f_j\}_{j\in \Nb}$ of continuous $\Rb$-valued functions on $\spec(|A|)$ such that $\max_{x\in \spec(|A|)}|x-x^3f_j(x)|\to 0$ as $j\to \infty$. Then
$$\left[\begin{matrix} |A|^3 & (|A|-|A|^3f_j(|A|))A^*C\\ & B \end{matrix}\right]\to \left[\begin{matrix} |A|^3 & \\ & B \end{matrix}\right]$$
in norm as $j\to \infty$.
Thus
$$\left[\begin{matrix} A & C \\ & B \end{matrix}\right]\succeq_\rC \left[\begin{matrix} |A|^3 & \\ & B \end{matrix}\right]\sim_\rC \left[\begin{matrix} |A| & \\ & B \end{matrix}\right]\sim_\rC \left[\begin{matrix} A & \\ & B \end{matrix}\right],$$
where for the two equivalences we apply Proposition~\ref{P-Cuntz}.
\end{proof}

From Lemma~\ref{L-M to C} we get the following.

\begin{proposition} \label{P-M to C}
We have an order-preserving homomorphism of monoids $\rW_\rM(\cA)\rightarrow \rW_\rC(\cA)$ sending $\langle A\rangle_\rM$ to $ \langle A\rangle_\rC$ for $A\in M(\cA)$. In particular, every dimension function for $\cA$ is a Sylvester matrix rank function.
\end{proposition}

\begin{remark} \label{R-M to C}
The homomorphism $\rW_\rM(\cA)\rightarrow \rW_\rC(\cA)$ in Proposition~\ref{P-M to C} may fail to be injective, as the following examples indicate.
For a compact Hausdorff space $X$ we denote by $C(X)$ the $C^*$-algebra of all $\Cb$-valued continuous functions on $X$.
\begin{enumerate}
\item Let $\cA=C([-1, 1])$, and let $f\in \cA$ be given by $f(x)=x$ for all $x\in [-1, 1]$. Then $f\not\in |f|\cA$ and $|f|\not\in f \cA$. From Lemma~\ref{L-commutative} we have $f\npreceq_\rM |f|$ and $|f|\npreceq_\rM f$, while $f\sim_\rC |f|$ by Proposition~\ref{P-Cuntz}.
\item Let $\cA=C([0, 1])$, and let $f\in \cA$ be given by $f(x)=xe^{i/x}$ for $0<x\le 1$ and $f(0)=0$. Then $f^*\not\in f\cA$ and $f\not\in f^* \cA$. From Lemma~\ref{L-commutative} we have $f\npreceq_\rM f^*$ and $f^*\npreceq_\rM f$, while $f\sim_\rC f^*$ by Proposition~\ref{P-Cuntz}.
\item Let $\cA=C(X)$ for some  connected compact Hausdorff space $X$ with more than one point. Take a nonnegative nonzero $f\in C(X)$ such that $f$ vanishes at some point of $X$. Then $f\not\in f^2\cA$. From Lemma~\ref{L-commutative} we have $f\npreceq_\rM f^2$, while $f\sim_\rC f^2$ by Proposition~\ref{P-Cuntz}.
\item Let $\cA=C([0, 1])$, and let $f, g\in \cA$ be given by $f(x)=x$ for all $x\in [0, 1]$ and $g(x)=x(2+\sin(1/x))/3$ for all $x\in (0, 1]$ and $g(0)=0$. Then $0\le g\le f$. Since $g\not \in f \cA$, from Lemma~\ref{L-commutative} we have $g\npreceq_\rM f$. Since $\supp(g)=\supp(f)$, we have $g\sim_\rC f$ by Proposition~\ref{P-Cuntz}.
\end{enumerate}
\end{remark}

Next we show that there are Sylvester matrix rank functions which are not dimension functions, even in the case $\cA$ is commutative.

\begin{example} \label{E-rk not dim}
Let $\cA=C([0, 1])$ and let $f\in \cA$ be the identity function given by $f(x)=x$ for all $x\in [0, 1]$. By Proposition~\ref{P-Cuntz}.(2) we have $f\sim_\rC f^2$. Note that $f^n\not\in \cA f^{n+1}$ for all $n\ge 0$. By Theorem~\ref{T-rk for square} there is a $\rk\in \Pb(\cA)$ such that
$\rk(f)=1/2$ and $\rk(f^2)=0$. In particular, $\rk$ is not a dimension function. We don't know whether $\rk$ must be regular or irregular. 
\end{example}

\begin{example} \label{E-rk not dim2}
Let $\cA=C([0, 1])$. Let $f\in \cA$ be the identity function given by $f(x)=x$ for all $x\in [0, 1]$, and let $g\in \cA$ be given by $g(x)=\frac{1}{\log (2/x)}$ for all $x\in (0, 1]$ and $g(0)=0$. Since $\supp(f)=\supp(g)=(0, 1]$, by Proposition~\ref{P-Cuntz}.(4) we have $f\sim_\rC g$. It is easily checked that $g^n\not\in \cA f$ for all $n\in \Nb$. By Lemma~\ref{L-not nilpotent}
there is some $\rk\in \Pb_{\rm reg}(\cA)$ such that $\rk(g)=1$ and $\rk(f)=0$. In particular, $\rk$ is not a dimension function.
\end{example}

The following lemma can be proved in the same way as Lemma~\ref{L-compare power1}.

\begin{lemma} \label{L-compare power}
Let $A\in M(\cA)_{\ge 0}$. Let $0<\delta\le \eta$ and $0<\theta$. Then
$$\diag(A^\eta, A^{\delta+\theta})\preceq_\rM \diag(A^\delta, A^{\eta+\theta}).$$ In particular,
$$\rk(A^{\delta+\theta})-\rk(A^{\eta+\theta})\le \rk(A^\delta)-\rk(A^\eta)$$ for all $\rk\in \Pb(\cA)$.
\end{lemma}

Despite the examples constructed in Examples~\ref{E-rk not dim} and \ref{E-rk not dim2}, the next result shows certain properties of dimension functions do hold for all Sylvester matrix rank functions.

\begin{proposition} \label{P-absolute value}
Let $\rk\in \Pb(\cA)$. The following hold:
\begin{enumerate}
\item  For any $A\in M(\cA)_{\ge 0}$, the function $f: (0, +\infty)\rightarrow \Rb$ sending $\theta$ to $\rk(A^\theta)$ is uniformly continuous.
\item  $\rk(A)=\rk(|A|)=\rk(A^*)$ and $\rk(|A|^\theta)=\rk(|A^*|^\theta)$
for all $A\in M(\cA)$ and $\theta>0$.
\item  $\rk(B^\theta) \le \rk(A^\theta )$ for all $0<\theta\le 2$ and $B, A\in M(\cA)$ with $0\le B\le A$.
\end{enumerate}
\end{proposition}
\begin{proof}
(1). Since $A^\eta\preceq_\rM A^\delta$ for all $0<\delta\le \eta$, $f$ is decreasing. Say, $A\in M_n(\cA)$. Then $\rk(A^\theta)\le n$ for all $\theta>0$. Put $L=\sup_{\theta>0}\rk(A^\theta)\le n$. Then $\rk(A^\theta)\to L$ as $\theta\to 0$.

Let $\varepsilon>0$. Then there is some $\zeta>0$ such that $L-\varepsilon<\rk(A^\delta)\le L$ for all $0<\delta<\zeta$.
For any $\theta>0$ and $0<\delta<\zeta$, taking $\max(\theta+\delta-\zeta, 0)<\kappa<\theta$, by Lemma~\ref{L-compare power} we have
$$0\le \rk(A^\theta)-\rk(A^{\theta+\delta})=\rk(A^{\kappa+ (\theta-\kappa)})-\rk(A^{\kappa+(\theta+\delta-\kappa)})\le \rk(A^{\theta-\kappa})-\rk(A^{\theta+\delta-\kappa})<\varepsilon.$$
Therefore $f$ is uniformly continuous.

(2). Let $A=S|A|$ be the polar decomposition of $A$. For any $0<\delta<1$, we have
$$ |A|^{1+\delta}=(S|A|^\delta)^*S|A|\preceq_\rM A=(S|A|^\delta)|A|^{1-\delta}\preceq_\rM |A|^{1-\delta}.$$
Thus
$$\rk(|A|^{1+\delta})\le \rk(A)\le \rk(|A|^{1-\delta}).$$ Letting $\delta\rightarrow 0$, by part (1) we get
$$\rk(|A|)\le \rk(A)\le \rk(|A|).$$ Thus $\rk(|A|)=\rk(A)$.

For any $0<\delta<\theta$, we have
$$ |A^*|^\theta=S|A|^\theta S^*=(S|A|^{\delta/2})|A|^{\theta-\delta}(|A|^{\delta/2}S^*)\preceq_\rM  |A|^{\theta-\delta}.$$
Thus $\rk(|A^*|^{\theta})\le \rk(|A|^{\theta-\delta})$. Letting $\delta\rightarrow 0$, by part (1) we get
$\rk(|A^*|^\theta)\le \rk(|A|^\theta)$. Replacing $A$ by $A^*$ we have $\rk(|A|^\theta)\le \rk(|A^*|^\theta)$. Therefore
$\rk(|A^*|^\theta)=\rk(|A|^\theta)$.
In particular,
$$\rk(A)=\rk(|A|)=\rk(|A^*|)=\rk(A^*).$$

(3). For any $0<\delta<1/2$, by \cite[Proposition 1.4.5]{Pedersen18} there is some $C_\delta\in M(\cA)$ such that $B^{1/2}=C_\delta A^\delta$. Then
$B=C_\delta A^{2\delta}C_\delta^*$, whence
$$B^2=C_\delta A^{2\delta} C_\delta^*C_\delta A^{2\delta} C_\delta^*\preceq_\rM A^{2\delta}C_\delta^* C_\delta A^{2\delta}=|C_\delta A^{2\delta}|^2.$$
By part (2) we have
$$\rk(B^2)\le \rk(|C_\delta A^{2\delta}|^2)=\rk(|(C_\delta A^{2\delta})^*|^2).$$
Note that
$$|(C_\delta A^{2\delta})^*|^2= C_\delta A^{4\delta}C_\delta^*\preceq_\rM A^{4\delta}.$$
Thus
$$\rk(B^2)\le \rk(|(C_\delta A^{2\delta})^*|^2)  \le  \rk(A^{4\delta}).$$
Letting $\delta\rightarrow 1/2$, by part (1) we get $\rk(B^2)\le \rk(A^2)$. For any $0<\theta\le 2$, we have $B^{\theta/2}\le A^{\theta/2}$ \cite[Proposition 1.3.8]{Pedersen18}. Therefore $\rk(B^\theta)\le \rk(A^\theta)$.
\end{proof}

For each unital ring $R$ and each $n\in \Nb$ denote by $\Pb_n(R)$ the set of $\rk\in \Pb(R)$ taking values in $\frac{1}{n}\Zb$.
It is a result of Schofield that when $R$ is an algebra over a field, $\bigcup_{n\in \Nb}\Pb_n(R)$ consists of exactly those $\rk$ in $\Pb(R)$ which are induced from unital homomorphisms from $R$ to $M_n(\Db)$ for some $n\in \Nb$ and some division ring $\Db$ and the unique Sylvester matrix rank function for $M_n(\Db)$ \cite[Theorem 7.12]{Schofield}.
In particular, when $R$ is an algebra over a field, one has $\bigcup_{n\in \Nb}\Pb_n(R)\subseteq \Pb_{\rm reg}(R)$.
It is an open question whether $\Pb(R)=\Pb_{\rm reg}(R)$ for algebras $R$ over a field \cite[Question 5.7]{JZ17}.

\begin{corollary} \label{C-not dense}
For $\cA=C([0, 1])$, $\bigcup_{n\in \Nb}\Pb_n(\cA)$ is not dense in $\Pb(\cA)$.
\end{corollary}
\begin{proof} Let $f\in \cA$ be the identity function given by $f(x)=x$ for all $x\in [0, 1]$. Let $\rk'\in \bigcup_{n\in \Nb}\Pb_n(\cA)$. By Proposition~\ref{P-absolute value}.(1) the function $(0, +\infty)\rightarrow \Rb$ sending $\theta$ to $\rk'(f^\theta)$ is continuous. Since $\rk'$ takes values in $\frac{1}{n}\Zb$ for some $n\in \Nb$, this means that $\rk'(f^\theta)$ does not depend on $\theta\in (0, +\infty)$. Then for any $\rk''$ in the closure of $\bigcup_{n\in \Nb}\Pb_n(\cA)$, one has that  $\rk''(f^\theta)$ does not depend on $\theta\in (0, +\infty)$. Thus the rank function $\rk$ in Example~\ref{E-rk not dim} does not lie in the closure of $\bigcup_{n\in \Nb}\Pb_n(\cA)$.
\end{proof}

From Proposition~\ref{P-absolute value} we know that every $\varphi\in \Pb(\cA)$ is determined by its restriction on $M(\cA)_{\ge 0}$.
It will be interesting to characterise elements of $\Pb(\cA)$ in terms of their restrictions on $M(\cA)_{\ge 0}$.

\begin{question} \label{Q-characterisation of rank function}
Let $\varphi: M(\cA)_{\ge 0}\rightarrow \Rb_{\ge 0}$. Assume that the following conditions hold:
\begin{enumerate}
\item  For any $A\in M(\cA)_{\ge 0}$, the function $f: (0, +\infty)\rightarrow \Rb$ sending $\theta$ to $\varphi(A^\theta)$ is continuous.
\item  $\varphi(|A|)=\varphi(|A^*|)$
for all $A\in M(\cA)$.
\item  $\varphi(B)\le \varphi(A)$ for all $B, A\in M(\cA)$ with $0\le B\le A$.
\item $\varphi(\diag(A, B))=\varphi(A)+\varphi(B)$ for all $A, B\in M(\cA)_{\ge 0}$.
\item $\varphi(0)=0$, $\varphi(1)=1$ and $\varphi(\lambda A)=\varphi(A)$ for all $\lambda\in (0, \infty)$ and $A\in M(\cA)_{\ge 0}$.
\end{enumerate}
Then is  $\rk: M(\cA)\rightarrow \Rb_{\ge 0}$ defined via $\rk(A)=\varphi(|A|)$ in $\Pb(\cA)$?
\end{question}

We say a function $\varphi: M(\cA)\rightarrow \Rb$ (resp. $M(\cA)_{\ge 0}\rightarrow \Rb$) is {\it lower semicontinuous} if it is lower semicontinuous on $M_{n, m}(\cA)$ (resp. $M_n(\cA)_{\ge 0}$) for all $n, m\in \Nb$ (resp. $n\in \Nb$), i.e. for any sequence $\{A_k\}_{k\in \Nb}$ in $M_{n, m}(\cA)$ (resp. $M_n(\cA)_{\ge 0}$) converging to some $A$ in norm, one has $\varphi(A)\le \varliminf_{k\to \infty}\varphi(A_k)$.
Every lower semicontinuous Sylvester matrix rank function $\rk$ for $\cA$ is a dimension function: for any $A, B\in M(\cA)$ with $A\preceq_\rC B$, taking two sequences $\{C_k\}_{k\in \Nb}$ and $\{D_k\}_{k\in \Nb}$ in $M(\cA)$ with $C_kBD_k\to A$ in norm as $k\to \infty$, one has $\rk(C_kBD_k)\le \rk(B)$ for each $k\in \Nb$, whence
$$ \rk(A)\le \varliminf_{k\to \infty}\rk(C_kBD_k)\le \rk(B).$$
From Proposition~\ref{P-M to C} we have the following consequence.

\begin{proposition} \label{P-lower}
The lower semicontinuous dimension functions for $\cA$ are exactly the lower semicontinuous Sylvester matrix rank functions for $\cA$.
\end{proposition}

Question~\ref{Q-characterisation of rank function} has affirmative answer for lower semicontinuous functions, even without assuming the condition (1).

\begin{lemma} \label{L-characterisation of lower rank function}
Let $\varphi: M(\cA)_{\ge 0}\rightarrow \Rb_{\ge 0}$ be lower semicontinuous satisfying the conditions (2)-(5) in Question~\ref{Q-characterisation of rank function}. Then $\rk: M(\cA)\rightarrow \Rb_{\ge 0}$ defined via $\rk(A)=\varphi(|A|)$ is a lower semicontinuous dimension function for $\cA$.
\end{lemma}
\begin{proof}
Let $A\in M_{m, n}(\cA)$ and $B\in M_{n, k}(\cA)$. We have $(AB)^*(AB)\le \|A\|^2B^*B$, whence
$$|AB|=((AB)^*(AB))^{1/2}\le (\|A\|^2B^*B)^{1/2}=\|A\| \cdot |B|$$ by \cite[Proposition 1.3.8]{Pedersen18}. Thus
$$\rk(AB)=\varphi(|AB|)\le \varphi(\|A\|\cdot |B|)=\varphi(|B|)=\rk(B)$$ when $A\neq 0$. Clearly $\rk(AB)=0\le \rk(B)$ when $A=0$. We also have $$\rk(A)=\varphi(|A|)=\varphi(|A^*|)=\rk(A^*),$$ and hence
$$\rk(AB)=\rk(B^*A^*)\le \rk(A^*)=\rk(A).$$

Since $\varphi$ is lower semicontinuous, so is $\rk$.

Let $A, B\in M(\cA)$ with $A\preceq_\rC B$. Then there are sequences $\{C_n\}_{n\in \Nb}$ and $\{D_n\}_{n\in \Nb}$ in $M(\cA)$ such that $C_nBD_n\to A$ in norm as $n\to \infty$. We have
$$\rk(A)\le \varliminf_{n\to \infty}\rk(C_nBD_n)\le \rk(B).$$

For any $A, B\in M(\cA)$, we have
$$\rk(\diag(A, B))=\varphi(\diag(|A|, |B|))=\varphi(|A|)+\varphi(|B|)=\rk(A)+\rk(B).$$  Clearly $\rk(1)=\varphi(1)=1$.
\end{proof}

Blackadar and Handelman note that for every dimension function $\varphi$ for $\cA$ there is a largest lower semicontinuous dimension function bounded above by $\varphi$ \cite[Proposition I.1.5]{BH} \cite[Proposition 6.4.3]{Blackadar88} \cite[Proposition 4.1]{Rordam}. Using their construction, here we extend this result to Sylvester matrix rank functions.

\begin{proposition} \label{P-lower dimension}
Let $\rk\in \Pb(\cA)$. Define $\rk': M(\cA)\rightarrow \Rb_{\ge 0}$ by $\rk'(A)=\sup_{\varepsilon>0}\rk((|A|-\varepsilon)_+)$. Then $\rk'$
is the largest lower semicontinuous dimension function for $\cA$ bounded above by $\rk$. In particular, if $\cA$ has a Sylvester rank function, then it has a lower semicontinuous dimension function.
\end{proposition}
\begin{proof} For any $\varepsilon>0$, since $(|A|-\varepsilon)_+\le |A|$, by Proposition~\ref{P-absolute value} one has
$$\rk((|A|-\varepsilon)_+)\le \rk(|A|)=\rk(A).$$ Thus $\rk'(A)\le \rk(A)$.

Let $A, B\in M(\cA)_{\ge 0}$ and $\varepsilon>0$ such that $B\ge A-\varepsilon/4$. We claim that
$$\rk'(B)\ge \rk((A-\varepsilon)_+).$$ Note that
$$(B-\varepsilon/4)_+\ge B-\varepsilon/4\ge A-\varepsilon/2.$$ Define continuous functions $f_{\varepsilon}, g_{\varepsilon}, h_{\varepsilon}: \Rb\rightarrow \Rb$ by $g_{\varepsilon}(x)=x-\varepsilon$ and $h_{\varepsilon}(x)=\max(0, x-\varepsilon)$ for all $x\in \Rb$, $f_{\varepsilon}(x)=0$ for $x\le \varepsilon$,  $f_{\varepsilon}(x)=(x-\varepsilon)^{1/2}$ for $x\in [\varepsilon, \varepsilon+1]$, and $f_{\varepsilon}(x)=1$ for $x\ge \varepsilon+1$. Then there is some constant $\lambda_{\varepsilon}>0$ such that $f^2_{\varepsilon}g_{\varepsilon/2}\ge \lambda_{\varepsilon} h_{\varepsilon}$. Thus
$$f_{\varepsilon}(A)(B-\varepsilon/4)_+f_{\varepsilon}(A)\ge f_{\varepsilon}(A)(A-\varepsilon/2)f_{\varepsilon}(A)=(f_{\varepsilon}g_{\varepsilon/2}f_{\varepsilon})(A)\ge \lambda_{\varepsilon} h_{\varepsilon}(A).$$
By Proposition~\ref{P-absolute value} we have
$$ \rk((B-\varepsilon/4)_+)\ge \rk(f_{\varepsilon}(A)(B-\varepsilon/4)_+f_{\varepsilon}(A))\ge \rk(\lambda_{\varepsilon} h_{\varepsilon}(A))=\rk(h_{\varepsilon}(A))=\rk((A-\varepsilon)_+).$$
Therefore
$$\rk'(B)\ge \rk((B-\varepsilon/4)_+)\ge \rk((A-\varepsilon)_+).$$
 This proves our claim.

Let $A\in M_n(\cA)_{\ge 0}$ and let $\delta>0$. Take $\varepsilon>0$ such that $\rk((A-\varepsilon)_+)\ge \rk'(A)-\delta$. Let $B\in M_n(\cA)_{\ge 0}$ with $\|A-B\|<\varepsilon/4$. Then $B\ge A-\varepsilon/4$. From the above claim we have
$$\rk'(B)\ge \rk((A-\varepsilon)_+)\ge \rk'(A)-\delta.$$ This shows that the restriction of $\rk'$ on $M(\cA)_{\ge 0}$ is lower semicontinuous.

Let $A, B\in M(\cA)$ with $0\le A\le B$. For every $\varepsilon>0$ we have $B\ge A-\varepsilon/4$, whence from the above claim we get $\rk'(B)\ge \rk((A-\varepsilon)_+)$. It follows that $\rk'(B)\ge \rk'(A)$. This verifies the condition (3) of Question~\ref{Q-characterisation of rank function} for $\rk'$.

Let $A\in M(\cA)$. Let $A=S|A|$ be the polar decomposition of $A$. For each $\varepsilon>0$, the element $B:=S(|A|-\varepsilon)_+$ is in $M(\cA)$ with
$|B|=(|A|-\varepsilon)_+$ and $|B^*|=(|A^*|-\varepsilon)_+$, whence by Proposition~\ref{P-absolute value} we have $$\rk((|A|-\varepsilon)_+)=\rk(|B|)=\rk(|B^*|)=\rk((|A^*|-\varepsilon)_+).$$
It follows that $\rk'(|A|)=\rk'(|A^*|)$.  This verifies the condition (2) of Question~\ref{Q-characterisation of rank function} for $\rk'$.

It is easily checked that the condition (4) and (5) of Question~\ref{Q-characterisation of rank function} hold for $\rk'$. From Lemma~\ref{L-characterisation of lower rank function} we conclude that $\rk'$ is a lower semicontinuous dimension function for $\cA$.

Now let $\rk''$ be a lower semicontinuous dimension function for $\cA$ such that $\rk''\le \rk$. Let $A\in M(\cA)$. Then $\rk''((|A|-\varepsilon)_+)\le \rk((|A|-\varepsilon)_+)$ for every $\varepsilon>0$. Since $\rk''$ is lower semicontinuous, we have
$$\rk''(A)=\rk''(|A|)\le \varliminf_{\varepsilon\to 0^+}\rk''((|A|-\varepsilon)_+)\le \varliminf_{\varepsilon\to 0^+}\rk((|A|-\varepsilon)_+)\le \rk'(A).$$
\end{proof}

It is a result of Blackadar and Handelman that there is a natural $1$-$1$ correspondence between lower semicontinuous dimension functions of $\cA$ and normalized $2$-quasitraces of $\cA$ \cite[page 327]{BH}. Haagerup showed that every $2$-quasitrace of a unital exact $C^*$-algebra is a trace \cite{Haagerup}. Thus, when $\cA$ is exact, there is a natural $1$-$1$ correspondence between lower semicontinuous dimension functions of $\cA$ and tracial states of $\cA$. Combining this with Proposition~\ref{P-lower dimension} we have the following consequence.

\begin{corollary} \label{C-trace}
If $\cA$ is exact and has a Sylvester rank function, then $\cA$ has a tracial state.
\end{corollary}

%%%%%%%%%%%%%%%%%%%%%%%%%%%%%%%%%%%%%%%%%%%%%%%%%%%%%%%%%%%%%%%%%%%%%%%%%%%%%%%%%%%%%%%%%%%%%%%%%%%%%%%%%%%%%%%%%%%%%%%%%%%%%%%%%%%%%%%%%%%%%%%%%%
\section{Finitely presented module Malcolmson semigroup} \label{S-fp Module M}

In this section we define the finitely presented module Malcolmson semigroup, give a natural isomorphism between the Grothendieck groups of the matrix Malcolmson semigroup and the finitely presented module Malcolmson semigroup for any unital ring in Theorem~\ref{T-matrix vs fp}, and show that for a von Neumann regular ring these two semigroups are naturally isomorphic in Theorem~\ref{T-regular}. Throughout this section, $R$ will be a unital ring. Denote by $\sFPM(R)$ the class of all finitely presented left $R$-modules. Note that any direct summand of a finitely presented module is finitely presented.

\begin{definition} \label{D-module Malcolmson}
For $\cM, \cN\in \sFPM(R)$, we write $\cM\lesssim \cN$ if we can write $\cN$ as $\cN_1\oplus \cN_2$ such that there is an exact sequence of left $R$-modules
$$\cN_1\rightarrow \cM\rightarrow \cN_2\rightarrow 0.$$
We say that $\cM$ is {\it Malcolmson subequivalent} to $\cN$, written as $\cM\preceq_\rM \cN$, if there are $n\in \Nb$ and  $\cM_1, \dots, \cM_n\in \sFPM(R)$ such that $\cM=\cM_1\lesssim \cM_2\lesssim\dots \lesssim \cM_n=\cN$.
This is a reflexive and transitive relation on $\sFPM(R)$. If $\cM\preceq_\rM \cN$ and $\cN\preceq_\rM \cM$, we say that $\cM$ and $\cN$ are {\it Malcolmson equivalent} and write $\cM\sim_\rM \cN$. Then $\sim_\rM$ is an equivalence relation on $\sFPM(R)$. We define the {\it finitely presented module Malcolmson semigroup} of $R$ as $\rV_\rM(R):=\sFPM(R)/\sim_\rM$.
For each $\cM\in \sFPM(R)$, denote by $\langle \cM\rangle_\rM$ the equivalence class of $\cM$ in $\rV_\rM(R)$.
\end{definition}

Note that if $\cM_1\preceq_\rM \cN_1$ and $\cM_2\preceq_\rM \cN_2$, then $\cM_1\oplus \cM_2\preceq_\rM \cN_1\oplus \cN_2$.  It follows that if $\cM_1\sim_\rM \cN_1$ and $\cM_2\sim_\rM \cN_2$, then
$\cM_1\oplus \cM_2\sim_\rM \cN_1\oplus \cN_2$. For any $\langle \cM\rangle_\rM, \langle \cN\rangle_\rM\in \rV_\rM(R)$, $\langle \cM \rangle_\rM+\langle \cN\rangle_\rM:=\langle \cM\oplus \cN\rangle_\rM$ does not depend on the choice of the representatives $\cM$ and $\cN$. Thus $\rV_\rM(R)$ is an abelian monoid with identity $0=\langle 0\rangle_\rM$ and  partial order given by $\langle \cM\rangle_\rM \preceq_\rM \langle \cN \rangle_\rM$ if $\cM\preceq_\rM \cN$. The element $\langle {}_RR \rangle_\rM$ is an order-unit of $(\rV_\rM(R), \preceq_\rM)$. The Sylvester module rank functions for $R$ can be identified with the states of $(\rV_\rM(R), \preceq_\rM, \langle {}_RR \rangle_\rM)$ naturally.

For each $\cM\in \sFPM(R)$, denote by $[\cM]_\rM$ the image of $\langle \cM\rangle_\rM$ in the Grothendieck group $\cG(\rV_\rM(R))$ of $\rV_\rM(R)$.
Similarly, for each $A\in M(R)$, denote by $[A]_\rM$ the image of $\langle A\rangle_\rM$ in the Grothendieck group $\cG(\rW_\rM(R))$ of $\rW_\rM(R)$.
By Lemma~\ref{L-scaled ordered} we have the partially ordered abelian groups with order-unit $(\cG(\rV_\rM(R)), \cG(\rV_\rM(R))_+, [{}_RR]_\rM)$  and $(\cG(\rW_\rM(R)), \cG(\rW_\rM(R))_+, [1]_{\rM})$. The following result explains the one-to-one correspondence between Sylvester matrix rank functions and Sylvester module rank functions in Theorem~\ref{T-bijection} at the Grothendieck group level.

\begin{theorem} \label{T-matrix vs fp}
There is a natural isomorphism
$$\Phi: (\cG(\rV_\rM(R)), \cG(\rV_\rM(R))_+, [{}_RR]_\rM)\rightarrow (\cG(\rW_\rM(R)), \cG(\rW_\rM(R))_+, [1]_\rM)$$ such that $$\Phi([R^m/R^nA]_\rM)=m[1]_\rM-[A]_\rM$$ and
$$\Phi^{-1}([A]_\rM)=m[{}_RR]_\rM-[R^m/R^nA]_\rM$$ for every $A\in M_{n, m}(R)$.
\end{theorem}

To prove Theorem~\ref{T-matrix vs fp}, we need to make some preparation. We describe the monoid morphism $\rV_\rM(R)\rightarrow \cG(\rW_\rM(R))$ first.

\begin{lemma} \label{L-module to matrix1}
Let $\cM\in \sFPM(R)$. Write $\cM$ as $R^m/R^nA$ for some $A\in M_{n, m}(R)$. Put
$$\phi(\cM)=m[1]_\rM-[A]_\rM\in \cG(\rW_\rM(R))_+.$$ Then $\phi(\cM)$ does not depend on the presentation of $\cM$ as $R^m/R^nA$.
\end{lemma}
\begin{proof} Put
$$f(A)=m[1]_\rM-[A]_\rM \in \cG(\rW_\rM(R))_+$$ for each $A\in M_{n, m}(R)$. Then it suffices to show $f(A)=f(B)$ whenever $A\in M_{n, m}(R)$ and $B\in M_{k, l}(R)$ with $R^m/R^nA \cong R^l/R^kB$.
By \cite[Lemma 2]{Malcolmson80} there is an isomorphism $\psi: R^m\oplus R^l\rightarrow R^m\oplus R^l$ such that
$$\psi(R^nA\oplus R^l)=R^m\oplus R^kB.$$
Then there is some invertible $C\in M_{m+l}(R)$ such that $\psi(u)=uC$ for all $u\in R^{m+l}=R^m\oplus R^l$. Note that $R^nA\oplus R^l=R^{n+l}\diag(A, I_l)$ and $R^m\oplus R^kB=R^{m+k}\diag(I_m, B)$. Thus
$R^{n+l}\diag(A, I_l)C=R^{m+k}\diag(I_m, B)$. Then there are some $D\in M_{n+l, m+k}(R)$ and $E\in M_{m+k, n+l}(R)$ such that $\diag(A, I_l)C=D\diag(I_m, B)$ and $E\diag(A, I_l)C=\diag(I_m, B)$. It follows that $\diag(A, I_l)\sim_\rM \diag(I_m, B)$. Therefore
$$[A]_\rM+l[1]_\rM=[\diag(A, I_l)]_\rM=[\diag(I_m, B)]_\rM=m[1]_\rM+[B]_\rM,$$
whence
$$f(A)=m[1]_\rM-[A]_\rM=l[1]_\rM-[B]_\rM=f(B).$$
\end{proof}

\begin{lemma} \label{L-module to matrix2}
Let $\cM, \cN\in \sFPM(R)$ with $\cM\preceq_\rM \cN$. Then $\phi(\cM) \preceq \phi(\cN)$.
\end{lemma}
\begin{proof} We may assume that $\cM\lesssim \cN$. Then we can write $\cN$ as $\cN_1\oplus \cN_2$ such that there is an exact sequence
$$\cN_1\rightarrow \cM\rightarrow \cN_2\rightarrow 0.$$

Write $\cM$ ($\cN_1$ resp.) as $R^m/R^nA$ ($R^l/R^kB$ resp.) for some $A\in M_{n, m}(R)$ ($B\in M_{k, l}(R)$ resp.). Then there is a homomorphism $\psi: R^l\rightarrow R^m$ such that the diagram
\begin{align*}
\xymatrix
{
& R^l \ar[d]  \ar[r]_{\psi} & R^m  \ar[d] \\
& \cN_1 \ar[r] & \cM
}
\end{align*}
commutes. Let $C\in M_{l, m}(R)$ such that $\psi(u)=uC$ for all $u\in R^l$. Then $\cN_2$ is isomorphic to $$R^m/(R^nA+R^lC)=R^m/(R^{n+l}\left[\begin{matrix} A \\ C \end{matrix}\right]),$$ whence $\cN$ is isomorphic to
$$R^l/R^kB\oplus R^m/(R^{n+l}\left[\begin{matrix} A \\ C \end{matrix}\right])\cong R^{l+m}/(R^{k+n+l}\left[\begin{matrix} B &\\ & A \\ & C \end{matrix}\right]).$$ Thus $\phi(\cM)=m[1]_\rM-[A]_\rM$ and
$$\phi(\cN)=(l+m)[1]_\rM-[\left[\begin{matrix} B &\\ & A \\ & C \end{matrix}\right]]_\rM,$$ and hence
$$ \phi(\cN)-\phi(\cM)=l[1]_\rM+[A]_\rM-[\left[\begin{matrix} B &\\ & A \\ & C \end{matrix}\right]]_\rM=[\left[\begin{matrix} I_l &\\ & A  \end{matrix}\right]]_\rM-[\left[\begin{matrix} B &\\ & A \\ & C \end{matrix}\right]]_\rM.$$
Note that $R^kBC=\psi(R^kB)\subseteq R^nA$. Thus there is some $D\in M_{k,n}(R)$ such that $BC=DA$. Then
\begin{align*}
\left[\begin{matrix} I_l &\\ & A  \end{matrix}\right]\gtrsim \left[\begin{matrix} B & -D \\ & I_n  \\I_l &  \end{matrix}\right]\cdot \left[\begin{matrix} I_l &\\ & A  \end{matrix}\right]\cdot \left[\begin{matrix} I_l & C\\ & I_m  \end{matrix}\right]=\left[\begin{matrix} B & \\ & A \\ I_l & C \end{matrix}\right]\succeq_\rM \left[\begin{matrix} B &\\ & A \\ & C \end{matrix}\right].
\end{align*}
Therefore $\phi(\cN)\succeq \phi(\cM)$.
\end{proof}

Next we describe the monoid morphism $\rW_\rM(R)\rightarrow \cG(\rV_\rM(R))$.
For $A\in M_{n, m}(R)$ put
$$\psi(A)=m[{}_RR]_\rM-[R^m/R^nA]_\rM\in \cG(\rV_\rM(R))_+.$$

\begin{lemma} \label{L-matrix to module}
Let $A, B\in M(R)$ with $A\preceq_\rM B$. Then $\psi(A)\preceq \psi(B)$.
\end{lemma}
\begin{proof} We may assume that $A\lesssim B$.
Then we may further assume that we are in one of the following three cases.

Case I.  $A=CB$ for some $C\in M(R)$. Say, $A\in M_{n, m}(R), B\in M_{k, m}(R)$, and $C\in M_{n, k}(R)$. Then $R^nA=R^nCB\subseteq R^kB$, whence $R^m/R^kB$ is a quotient module of $R^m/R^nA$. Thus $\langle R^m/R^kB\rangle_\rM\preceq_\rM \langle R^m/R^nA\rangle_\rM$. Hence
$$\psi(A)=m[{}_RR]_\rM-[R^m/R^nA]_\rM\preceq m[{}_RR]_\rM-[R^l/R^kB]_\rM=\psi(B).$$

Case II. $A=BD$ for some $D\in M(R)$. Say, $A\in M_{n, m}(R), B\in M_{n, l}(R)$, and $D\in M_{l, m}(R)$. We have
\begin{align*}
\psi(B)-\psi(A)&=l[{}_RR]_\rM-[R^l/R^n B]_\rM-m[{}_RR]_\rM+[R^m/R^nA]_\rM\\
&=[R^l\oplus (R^m/R^nA)]_\rM-[R^m\oplus (R^l/R^nB)]_\rM.
\end{align*}
Denote by $\alpha$ the homomorphism $R^m\oplus (R^l/R^nB)\rightarrow R^m/R^nA$ sending $(u, v+R^nB)$ to $u-vD+R^nA$. Then $\alpha$ is surjective.
Also denote by $\beta$ the homomorphism $R^l\rightarrow R^m\oplus (R^l/R^nB)$ sending $w$ to $(wD, w+R^nB)$. Then $\alpha\beta=0$. Let $(u, v+R^nB)\in \ker \alpha$. Then $u-vD=xA$ for some $x\in R^n$. We have
$$ \beta(xB+v)=(xBD+vD, xB+v+R^nB)=(xA+vD, v+R^nB)=(u, v+R^nB).$$
Thus $\im(\beta)=\ker \alpha$. This means that we have an exact sequence
$$ R^l\overset{\beta}{\rightarrow} R^m\oplus (R^l/R^nB)\overset{\alpha}{\rightarrow} R^m/R^nA\rightarrow 0.$$
Therefore
$$\langle R^m\oplus (R^l/R^nB)\rangle_\rM \preceq_\rM \langle R^l\oplus (R^m/R^nA)\rangle_\rM.$$ Thus $\psi(A)\preceq \psi(B)$.

Case III. $B=\left[\begin{matrix} C & E\\ & D \end{matrix}\right]$ and $A=\left[\begin{matrix} C & \\ & D \end{matrix}\right]$ for some $C, D, E\in M(R)$.
Say, $C\in M_{n, m}(R), E\in M_{n, l}(R)$, and $D\in M_{k, l}(R)$. We have
$$ \psi(B)-\psi(A)=[R^{m+l}/R^{n+k}A]_\rM-[R^{m+l}/R^{n+k}B]_\rM.$$
Note that
$$R^{m+l}/R^{n+k}A\cong (R^m/R^nC)\oplus (R^l/R^kD).$$
Denote by $\alpha$ the homomorphism $R^{m+l}/R^{n+k}B\rightarrow R^m/R^nC$ sending $(u, v)+R^{n+k}B$ to $u+R^nC$. Then $\alpha$ is surjective. Also denote by $\beta$ the homomorphism $R^l/R^kD\rightarrow R^{m+l}/R^{n+k}B$ sending $w+R^kD$ to $(0, w)+R^{n+k}B$. Then $\alpha\beta=0$. Let $(u, v)+R^{n+k}B\in \ker \alpha$. Then $u=xC$ for some $x\in R^n$. We have
$$ \beta(v-xE+R^kD)=(0, v-xE)+R^{n+k}B=(0, v-xE)+(x, 0)B+R^{n+k}B=(u, v)+R^{n+k}B.$$
Thus $\im(\beta)=\ker \alpha$. This means that we have an exact sequence
$$ R^l/R^kD\overset{\beta}{\rightarrow} R^{m+l}/R^{n+k}B \overset{\alpha}{\rightarrow} R^m/R^nC\rightarrow 0.$$
Therefore
$$\langle R^{m+l}/R^{n+k}B\rangle_\rM \preceq_\rM \langle R^{m+l}/R^{n+k}A\rangle_\rM.$$
Thus $\psi(A)\preceq \psi(B)$.
\end{proof}

We are ready to prove Theorem~\ref{T-matrix vs fp}.

\begin{proof}[Proof of Theorem~\ref{T-matrix vs fp}]
From Lemma~\ref{L-module to matrix2} we have a well-defined map $\phi': \rV_\rM(R)\rightarrow \cG(\rW_\rM(R))$ sending $\langle \cM \rangle_\rM$ to $\phi(\cM)$, and for any $\langle \cM \rangle_\rM\preceq_\rM \langle \cN \rangle_\rM$, we have $\phi'(\langle \cM \rangle_\rM)\preceq \phi'(\langle \cN \rangle_\rM)$. Clearly $\phi'$ is additive. It induces a group homomorphism
$$\Phi: \cG(\rV_\rM(R))\rightarrow \cG(\rW_\rM(R))$$ satisfying $\Phi([{}_RR]_\rM)=[1]_\rM$. By Lemma~\ref{L-positive cone} $\Phi$ is positive, i.e. $\Phi(\cG(\rV_\rM(R))_+)\subseteq \cG(\rW_\rM(R))_+$.

From Lemma~\ref{L-matrix to module} we have a well-defined additive map $\psi': \rW_\rM(R)\rightarrow \cG(\rV_\rM(R))$ sending $\langle A \rangle_\rM$ to $\psi(A)$, and for any $\langle A \rangle_\rM\preceq_\rM \langle B\rangle_\rM$, we have $\psi'(\langle A \rangle_\rM)\preceq \psi'(\langle B \rangle_\rM)$. It induces a group homomorphism $\Psi: \cG(\rW_\rM(R))\rightarrow \cG(\rV_\rM(R))$ satisfying $\Psi([1]_\rM)=[{}_RR]_\rM$. By Lemma~\ref{L-positive cone} $\Psi$ is positive, i.e. $\Phi(\cG(\rW_\rM(R))_+)\subseteq \cG(\rV_\rM(R))_+$.

Clearly $\Phi$ and $\Psi$ are inverses to each other. Thus they are isomorphisms between $(\cG(\rV_\rM(R)), \cG(\rV_\rM(R))_+, [{}_RR]_\rM)$ and $(\cG(\rW_\rM(R)), \cG(\rW_\rM(R))_+, [1]_\rM)$.
\end{proof}

\begin{remark} \label{R-matrix vs fp}
Let $S$ be another unital ring and let $f: R\rightarrow S$ be a unital ring homomorphism. Then we have a natural monoid map $\rW_\rM(R)\rightarrow \rW_\rM(S)$ sending $\langle A\rangle_\rM$ to $\langle f(A)\rangle_\rM$. It induces a homomorphism
$$(\cG(\rW_\rM(R)), \cG(\rW_\rM(R))_+, [1]_\rM)\rightarrow (\cG(\rW_\rM(S)), \cG(\rW_\rM(S))_+, [1]_\rM).$$  We also have a monoid map $\rV_\rM(R)\rightarrow \rV_\rM(S)$ sending $\langle \cM\rangle_\rM$ to $\langle S\otimes_R \cM\rangle_\rM$. It induces a homomorphism
$$(\cG(\rV_\rM(R)), \cG(\rV_\rM(R))_+, [{}_RR]_\rM) \rightarrow (\cG(\rV_\rM(S)), \cG(\rV_\rM(S))_+, [{}_SS]_\rM).$$
Clearly
the diagram
\begin{align*}
\xymatrix
{
& (\cG(\rW_\rM(R)), \cG(\rW_\rM(R))_+, [1]_\rM) \ar[d]  \ar[r] & (\cG(\rW_\rM(S)), \cG(\rW_\rM(S))_+, [1]_\rM) \ar[d] \\
&  (\cG(\rV_\rM(R)), \cG(\rV_\rM(R))_+, [{}_RR]_\rM)\ar[r] & (\cG(\rV_\rM(S)), \cG(\rV_\rM(S))_+, [{}_SS]_\rM)
}
\end{align*}
commutes.
\end{remark}

When $R$ is von Neumann regular, the isomorphism in Theorem~\ref{T-matrix vs fp} can be lifted to an isomorphism between  $(\rV_\rM(R), \preceq_\rM, \langle {}_RR\rangle_\rM)$ and $(\rW_\rM(R), \preceq_\rM, \langle 1\rangle_\rM)$.

\begin{theorem} \label{T-regular}
Assume that  $R$ is von Neumann regular. There is a natural isomorphism $\widetilde{\Psi}: (\rW_\rM(R), \preceq_\rM, \langle 1\rangle_\rM) \rightarrow (\rV_\rM(R), \preceq_\rM, \langle {}_RR\rangle_\rM)$ given by $\widetilde{\Psi}(\langle A\rangle_{\rM})=\langle R^n A\rangle_\rM$ for every $A\in M_{n, m}(R)$.
\end{theorem}

We need some preparation for the proof of Theorem~\ref{T-regular} which will be given at the end of this section. Lemmas~\ref{L-order for proj} and \ref{L-regular surjective} below should be known. For completeness, we give a proof. 

\begin{lemma} \label{L-order for proj}
Let $P\in M_n(R)$ and $Q\in M_k(R)$ be idempotents. The following are equivalent:
\begin{enumerate}
\item there are $C\in M_{n, k}(R)$ and $D\in M_{k, n}(R)$ such that $P=CQD$;
\item $R^nP$ is isomorphic to a direct summand of $R^kQ$.
\end{enumerate}
\end{lemma}
\begin{proof}
(1)$\Rightarrow$(2). Replacing $C$ and $D$ by $PCQ$ and $QDP$ respectively if necessary, we may assume that $PCQ=C$ and $D=QDP$.
Define a homomorphism $\varphi: R^nP\rightarrow R^kQ$ by $\varphi(x)=xC$. Also define a homomorphism $\psi: R^kQ\rightarrow R^nP$ by $\psi(y)=yD$. For each $x\in R^nP$ we have
$$\psi(\varphi(x))=xCD=xCQD=xP=x.$$ Therefore $\varphi$ is an embedding of $R^nP$ into $R^kQ$, and $R^kQ=\varphi(R^nP)\oplus \ker \psi$.

(2)$\Rightarrow$(1). Let $\varphi$ be the embedding $R^nP\rightarrow R^k Q$, and let $\psi$ be a homomorphism $R^kQ\rightarrow R^nP$ such that $\psi\circ \varphi$ is the identity map of $R^n P$. Since $R^nP$ is a direct summand of $R^n$, we may extend $\varphi$ to a homomorphism $R^n\rightarrow R^kQ$, whence find a $C\in M_{n, k}(R)$ such that $CQ=C$ and $xC=\varphi(x)$ for all $x\in R^nP$. Replacing $C$ by $PC$, we may assume that $C=PC$. Similarly, we can find a
$D\in M_{k, n}(R)$ such that $DP=D=QD$ and $yD=\psi(y)$ for all $y\in R^kQ$. Then
$$xCD=\psi(xC)=\psi(\varphi(x))=x$$ for all $x\in R^nP$. Thus $zPCD=zP$ for all $z\in R^n$, and hence $PCD=P$. Then
$$CQD=CD=PCD=P.$$
\end{proof}

\begin{lemma} \label{L-regular surjective}
Assume that $R$ is von Neumann regular. Let $A\in M_{n, m}(R)$. Then there are some $k\in \Nb$ and an idempotent $P\in M_k(R)$ satisfying the following conditions:
\begin{enumerate}
\item there are $B\in M_{k, n}(R), C\in M_{m, k}(R), D\in M_{n, k}(R)$ and $E\in M_{k, m}(R)$ such that $P=BAC$ and $A=DPE$;
\item $R^nA$ and $R^kP$ are isomorphic as left $R$-modules.
\end{enumerate}
\end{lemma}
\begin{proof} We consider first the case $n\ge m$.
Set $k=n$. Set
$$C=\left[\begin{matrix} I_m & 0_{m, n-m} \end{matrix}\right]\in M_{m, n}(R),$$ and
$$E=\left[\begin{matrix} I_m \\ 0_{n-m, m} \end{matrix}\right]\in M_{n, m}(R).$$ Also set $A'=AC\in M_n(R)$. Then $A=A'E$, and $R^nA\cong R^nA'$.

Since $R$ is von Neumann regular, so is $M_n(R)$ \cite[Exercise 6.20]{Lam01} \cite[Theorem 1.7]{Goodearl}. Thus there is some $B\in M_n(R)$ such that $A'BA'=A'$. Then $P:=BA'=BAC$ is an idempotent. Note that
$$R^nA'=R^nA'BA'\subseteq R^nBA'\subseteq R^nA'.$$ Thus
$$R^nA\cong R^nA'=R^nBA'=R^nP.$$
Also
$$A'PE=A'BA'E=A'E=A.$$

Next we consider the case $n<m$. Set $k=m$. Set
$$B=\left[\begin{matrix} I_n \\ 0_{m-n, n} \end{matrix}\right]\in M_{m, n}(R),$$ and
$$D=\left[\begin{matrix} I_n & 0_{n, m-n} \end{matrix}\right]\in M_{n, m}(R).$$
Also set $A'=BA\in M_m(R)$. Then $A=DA'$, and $R^nA=R^mA'$.

As noted above, $M_n(R)$ is von Neumann regular. Thus there is some $C\in M_m(R)$ such that $A'CA'=A'$. Then
$$P:=CA'=BCA=(BC)AI_m$$ is an idempotent.
Note that
$$R^mA'=R^mA'CA'\subseteq R^mCA'=R^mP\subseteq R^mA'.$$ Thus
$$R^nA=R^mA'=R^mP.$$
Finally,
$$(DA')PI_m=DA'P=DA'CA'=DA'=A.$$
\end{proof}

\begin{lemma} \label{L-order for regular}
Assume that $R$ is von Neumann regular. Let $A\in M_{n, m}(R)$ and $B\in M_{k, l}(R)$. Then the following are equivalent:
\begin{enumerate}
\item $A\preceq_\rM B$;
\item there are $C\in M_{n, k}(R)$ and $D\in M_{l, m}(R)$ such that $A=CBD$;
\item $R^nA$ is isomorphic to a direct summand of $R^kB$;
\item $R^nA$ is isomorphic to a quotient module of $R^kB$.
\end{enumerate}
\end{lemma}
\begin{proof}
(1)$\Rightarrow$(2). We may assume that $B=\left[\begin{matrix} C & E\\ & D \end{matrix}\right]$ and $A=\left[\begin{matrix} C & \\ & D \end{matrix}\right]$ for some $C, D, E\in M(R)$. By Lemma~\ref{L-regular surjective} there are some $U, V, X, Y\in M(R)$ such that $P:=UCV$ is an idempotent and $C=XPY$. Say, $D\in M_{i, j}(R)$. Then
\begin{align*}
\left[\begin{matrix} PU & \\ & I_i \end{matrix}\right]\left[\begin{matrix} C & E\\ & D \end{matrix}\right]\left[\begin{matrix} V & \\ & I_j \end{matrix}\right]=\left[\begin{matrix} P & PUE\\ & D \end{matrix}\right],
\end{align*}
and
\begin{align*}
\left[\begin{matrix} X & \\ & I_i \end{matrix}\right]\left[\begin{matrix} P & PUE\\ & D \end{matrix}\right]\left[\begin{matrix} P & -UE\\ & I_j \end{matrix}\right]\left[\begin{matrix} Y & \\ & I_j \end{matrix}\right]=\left[\begin{matrix} X & \\ & I_i \end{matrix}\right]\left[\begin{matrix} P & \\ & D \end{matrix}\right]\left[\begin{matrix} Y & \\ & I_j \end{matrix}\right]=\left[\begin{matrix} C & \\ & D \end{matrix}\right].
\end{align*}
Therefore
\begin{align*}
\left[\begin{matrix} X & \\ & I_i \end{matrix}\right]\left[\begin{matrix} PU & \\ & I_i \end{matrix}\right]\left[\begin{matrix} C & E\\ & D \end{matrix}\right]\left[\begin{matrix} V & \\ & I_j \end{matrix}\right]\left[\begin{matrix} P & -UE\\ & I_j \end{matrix}\right]\left[\begin{matrix} Y & \\ & I_j \end{matrix}\right]=\left[\begin{matrix} C & \\ & D \end{matrix}\right].
\end{align*}

(2)$\Rightarrow$(1) and (3)$\Rightarrow$(4) are trivial.

(2)$\Leftrightarrow$(3) follows from Lemmas~\ref{L-regular surjective} and \ref{L-order for proj}.

(4)$\Rightarrow$(3). From Lemma~\ref{L-regular surjective} we know that $R^nA$ is projective.
\end{proof}

\begin{lemma} \label{L-order for module regular}
Assume that $R$ is von Neumann regular. Let $\cM, \cN\in \sFPM(R)$. Then $\cM\preceq_\rM \cN$ if and only if $\cM$ is isomorphic to a quotient module of $\cN$.
\end{lemma}
\begin{proof} If $\cM$ is isomorphic to a quotient module of $\cN$, then clearly $\cM\lesssim \cN$, whence $\cM\preceq_\rM \cN$. This proves the ``if'' part.

To prove the ``only if'' part, we may assume that $\cM\lesssim \cN$. Then we can write $\cN$ as $\cN_1\oplus \cN_2$ such that there is a homomorphism $\varphi: \cN_1\rightarrow \cM$ with $\cM/\im(\varphi)\cong \cN_2$. Since $R$ is von Neumann regular, every finitely presented left $R$-module is projective \cite[Exercise 6.19]{Lam01} \cite[Theorem 1.11]{Goodearl}. Thus $\cM\cong \im(\varphi)\oplus \cN_2$. It follows that $\cM$ is isomorphic to a quotient module of $\cN$.
\end{proof}

For von Neumann regular rings, every finitely presented left module is projective \cite[Exercise 6.19]{Lam01} \cite[Theorem 1.11]{Goodearl}. Thus Theorem~\ref{T-regular} follows from Lemmas~\ref{L-regular surjective}, \ref{L-order for regular} and \ref{L-order for module regular}.

\begin{remark} \label{R-projective}
Denote by $\rV(R)$ the set of isomorphism classes of finitely generated projective left $R$-modules. 
For each finitely generated projective left $R$-module $\cM$, denote by $\langle \cM\rangle$ its isomorphism class in $\rV(R)$.  Then $\rV(R)$ is an abelian monoid with addition given by $\langle \cM\rangle+\langle \cN\rangle:=\langle \cM\oplus \cN\rangle$.
There is a natural monoid homomorphism $\Theta: \rV(R)\rightarrow \rV_\rM(R)$ sending $\langle \cM\rangle $ to $\langle \cM\rangle_\rM$. 
If $R$ is stably finite, then $\rV(R)$ has the partial order given by $\langle \cM\rangle\preceq \langle \cN\rangle$ if $\cM$ is isomorphic to a direct summand of $\cN$, and an order-unit $\langle {}_RR\rangle$, and $\Theta$ preserves the partial order and order-unit. 
When $R$ is stably finite and von Neumann regular, since every finitely presented left $R$-module is projective \cite[Exercise 6.19]{Lam01} \cite[Theorem 1.11]{Goodearl}, it follows from Lemma~\ref{L-order for module regular} that $\Theta$ is an isomorphism of partially ordered abelian monoids with order-units. 
\end{remark}

%%%%%%%%%%%%%%%%%%%%%%%%%%%%%%%%%%%%%%%%%%%%%%%%%%%%%%%%%%%%%%%%%%%%%%%%%%%%%%%%%%%%%%%%%%%%%%%%%%%%%%%%%%%%%%%%%%%%%%%%%%%%%%%%%%%%%%%%%%%%%%%%%%
\section{Finitely generated module Malcolmson semigroup} \label{S-fg Module M}

In this section we define the finitely generated module Malcolmson semigroup and show that for any unital ring the finitely presented module Malcolmson semigroup embeds naturally into the finitely generated module Malcolmson semigroup in Theorem~\ref{T-fp vs fg}. Throughout this section
$R$ will be a unital ring. Denote by $\sFGM(R)$ the class of all finitely generated left $R$-modules.

\begin{definition} \label{D-fg module Malcolmson}
For $\cM, \cN\in \sFGM(R)$, we define the relations $\cM\lesssim \cN$, $\cM\preceq_{\rM}^{\rm fg} \cN$, and $\cM\sim_\rM^{\rm fg} \cN$ in the same way as 
$\cM\lesssim \cN$, $\cM\preceq_{\rM} \cN$, and $\cM\sim_\rM \cN$ in Definition~\ref{D-module Malcolmson}, except that all modules appearing are in $\sFGM(R)$ now. We define the {\it finitely generated module Malcolmson semigroup} of $R$ as $\rU_\rM(R):=\sFGM(R)/\sim_\rM^{\rm fg}$.
For each $\cM\in \sFGM(R)$, denote by $\langle \cM\rangle_\rM^{\rm fg}$ the equivalence class of $\cM$ in $\rU_\rM(R)$.
\end{definition}

Similar to the situation of $\rV_\rM(R)$, 
for any $\langle \cM\rangle_\rM^{\rm fg}, \langle \cN\rangle_\rM^{\rm fg}\in \rU_\rM(R)$, $\langle \cM \rangle_\rM^{\rm fg}+\langle \cN\rangle_\rM^{\rm fg}:=\langle \cM\oplus \cN\rangle_\rM^{\rm fg}$ does not depend on the choice of the representatives $\cM$ and $\cN$. Thus $\rU_\rM(R)$ is an abelian monoid with identity $0=\langle 0\rangle_\rM^{\rm fg}$ and  partial order given by $\langle \cM\rangle_\rM^{\rm fg} \preceq_\rM^{\rm fg} \langle \cN \rangle_\rM^{\rm fg}$ if $\cM\preceq_\rM^{\rm fg} \cN$. The element $\langle {}_RR\rangle_{\rM}^{\rm fg}$ is an order-unit of $(\rU_\rM(R), \preceq_\rM^{\rm fg})$.

\begin{remark} \label{R-fg vs fp}
For any $\cM\in \sFGM(R)$, writing  $\cM$ as a quotient module of $R^n$ for some $n\in \Nb$, we see that $R^n$ is finitely presented and $\cM \lesssim R^n$. Thus, $\cM\in \sFGM(R)$ and $\cN\in \sFPM(R)$ with $\cM\lesssim \cN$ does not imply that $\cM\in \sFPM(R)$.
\end{remark}

\begin{theorem} \label{T-fp vs fg}
$(\rV_\rM(R), \preceq_\rM, \langle {}_RR\rangle_\rM)$ is naturally a partially ordered submonoid of $(\rU_\rM(R), \preceq_\rM^{\rm fg}, \langle{}_RR\rangle_\rM^{\rm fg})$.
\end{theorem}

Theorem~\ref{T-fp vs fg} follows from Lemma~\ref{L-fp vs fg} below.

\begin{lemma} \label{L-presentation1}
Let $\cN\in \sFPM(R)$ and let $\cM$ be a quotient module of $\cN$. Write $\cM$ as $\cM_1\oplus \cM_2$. Then $\cN$ has a quotient module $\widetilde{\cN}\in \sFPM(R)$ such that we can write $\widetilde{\cN}$ as $\widetilde{\cN}_1\oplus \widetilde{\cN}_2$ and $\cM_i$ is a quotient module of $\widetilde{\cN}_i$ for $i=1, 2$.
\end{lemma}
\begin{proof} Take a surjective homomorphism $\varphi: R^n\rightarrow \cN$ for some $n\in \Nb$. Then $\ker \varphi$ is finitely generated. Denote by $\pi$ the quotient map $\cN\rightarrow \cM$.
Then $\ker \varphi\subseteq \ker (\pi\varphi)\subseteq (\pi\varphi)^{-1}(\cM_1)\cap (\pi\varphi)^{-1}(\cM_2)$.

For any $u\in R^n$, writing $\pi\varphi(u)$ as $x_1+x_2$ with $x_i\in \cM_i$ for $i=1,2$ and taking $v_i\in R^n$ with $\pi\varphi(v_i)=x_i$ for $i=1,2$, we have $\pi\varphi(v_1+v_2)=\pi\varphi(u)$, whence
$$u\in v_1+v_2+\ker(\pi\varphi)\subseteq (\pi\varphi)^{-1}(\cM_1)+(\pi\varphi)^{-1}(\cM_2).$$ Therefore $$R^n=(\pi\varphi)^{-1}(\cM_1)+(\pi\varphi)^{-1}(\cM_2).$$

Take a finite generating subset $W$ of $R^n$. Write each $w\in W$ as $w_1+w_2$ with $w_i\in (\pi\varphi)^{-1}(\cM_i)$ for $i=1,2$. Denote by $\cW_i$ the submodule of $R^n$ generated by $w_{3-i}$ for $w\in W$ and $\ker \varphi$.  Then $\cW_i$ is finitely generated. Since $\cW_1+\cW_2=R^n$, it is easily checked that the homomorphism $R^n\rightarrow (R^n/\cW_1)\oplus (R^n/\cW_2)$ sending $u$ to $(u+\cW_1, u+\cW_2)$ is surjective and has kernel $\cW_1\cap \cW_2$. Thus we have
$$R^n/(\cW_1\cap \cW_2)\cong (R^n/\cW_1)\oplus (R^n/\cW_2).$$ Put $\widetilde{\cN}=R^n/(\cW_1\cap \cW_2)$, and $\widetilde{\cN}_i=R^n/\cW_i$ for $i=1, 2$. Since $\widetilde{\cN}_1$ and $\widetilde{\cN}_2$ are finitely presented, so is $\widetilde{\cN}$. As $\ker \varphi\subseteq \cW_1\cap \cW_2$, $\widetilde{\cN}$ is a quotient module of $\cN$. For each $i=1, 2$, as
$$\cW_i\subseteq (\pi\varphi)^{-1}(\cM_{3-i})+\ker \varphi=(\pi\varphi)^{-1}(\cM_{3-i}),$$
$\cM_i\cong R^n/(\pi\varphi)^{-1}(\cM_{3-i})$ is a quotient module of $\widetilde{\cN}_i$.
\end{proof}

\begin{lemma} \label{L-presentation2}
Let $\cN\in \sFPM(R)$ and $\cM', \cM\in \sFGM(R)$ such that $\cM'$ is a quotient module of $\cN$ and $\cM\lesssim \cM'$. Then there are some $\widetilde{\cN}, \cN'\in \sFPM(R)$ such that $\cN'\lesssim \widetilde{\cN}$, and that $\widetilde{\cN}$ and $\cM$ are quotient modules of $\cN$ and $\cN'$ respectively.
\end{lemma}
\begin{proof} Write $\cM'$ as $\cM'_1\oplus \cM'_2$ such that there is an exact sequence
$$\cM'_1\overset{\theta}{\rightarrow} \cM\overset{\zeta}{\rightarrow} \cM'_2\rightarrow 0.$$
By Lemma~\ref{L-presentation1} we can find a quotient module $\widetilde{\cN}\in \sFPM(R)$ of $\cN$ such that we can write $\widetilde{\cN}$ as $\widetilde{\cN}_1\oplus \widetilde{\cN}_2$ and $\cM_i'$ is a quotient module of $\widetilde{\cN}_i$ for $i=1, 2$.

Let $i=1, 2$. Take a surjective homomorphism $\varphi_i: R^{n_i}\rightarrow \widetilde{\cN}_i$ for some $n_i\in \Nb$. Then $\ker \varphi_i$ is finitely generated. Denote by $\pi_i$ the quotient map $\widetilde{\cN}_i\rightarrow \cM'_i$.

Take a homomorphism $\alpha: R^{n_2}\rightarrow \cM$ such that $\zeta\alpha=\pi_2\varphi_2$. Denote by $\beta$ the homomorphism $R^{n_1+n_2}=R^{n_1}\oplus R^{n_2}\rightarrow \cM$ sending $(u, v)$ to $(\theta \pi_1\varphi_1)(u)+\alpha(v)$. Then $\beta$ is surjective.

Take a finite generating subset $W_2$ of $\ker \varphi_2$. For each $w\in W_2$, since $\zeta\alpha(w)=\pi_2\varphi_2(w)=0$, we have $\alpha(w)\in \ker \zeta=\im(\theta \pi_1\varphi_1)$, whence we can find some $x_w\in R^{n_1}$ such that $\theta\pi_1\varphi_1(x_w)=\alpha(w)$. Denote by $\cW$ the submodule of $R^{n_1}\oplus R^{n_2}$ generated by $(-x_w, w)$ for $w\in W_2$. Then $\cW$ is finitely generated and contained in $\ker \beta$.

Put
$$\cN'=(R^{n_1}\oplus R^{n_2})/(\ker \varphi_1+\cW).$$ Then $\cN'$ is finitely presented. Since
$$\ker \varphi_1+\cW\subseteq \ker(\theta \pi_1\varphi_1)+\ker \beta\subseteq \ker \beta,$$
$\cM$ is a quotient module of $\cN'$.

Denote by $\gamma$ the homomorphism $\widetilde{\cN}_1\cong R^{n_1}/\ker \varphi_1\rightarrow \cN'$ sending $u+\ker \varphi_1$ to $(u, 0)+\ker \varphi_1+\cW$. Also denote by $\xi$ the homomorphism $\cN'\rightarrow R^{n_2}/\ker \varphi_2\cong \widetilde{\cN_2}$ sending $(u, v)+\ker \varphi_1+\cW$ to $v+\ker \varphi_2$. Then $\xi$ is surjective and $\xi\gamma=0$. For any $(u, v)+\ker \varphi_1+\cW\in \ker \xi$, we have $(x, v)\in \cW$ for some $x\in R^{n_1}$, whence
$$(u, v)+\ker \varphi_1+\cW=(u-x, 0)+\ker \varphi_1+\cW\in \im(\gamma).$$
Thus the sequence
$$\widetilde{\cN}_1\overset{\gamma}{\rightarrow} \cN'\overset{\xi}{\rightarrow} \widetilde{\cN}_2\rightarrow 0$$
is exact. Therefore $\cN'\lesssim \widetilde{\cN}_1\oplus \widetilde{\cN}_2=\widetilde{\cN}$.
\end{proof}

\begin{lemma} \label{L-fp vs fg}
Let $\cM, \cN\in \sFPM(R)$. Then $\cM\preceq_\rM \cN$ if and only if $\cM\preceq_\rM^{\rm fg} \cN$.
\end{lemma}
\begin{proof} The ``only if'' part is trivial.

Assume that $\cM\preceq_\rM^{\rm fg} \cN$. Then there are $n\in \Nb$ and $\cM_1, \dots, \cM_n\in \sFGM(R)$ such that $\cN=\cM_1\gtrsim \cM_2\gtrsim \cdots\gtrsim \cM_n=\cM$. Applying Lemma~\ref{L-presentation2} to $(\cN, \cM', \cM)=(\cN, \cN, \cM_2)$, we find $\cN_2,\cN_2'\in \sFPM(R)$ such that $\cN_2'\lesssim \cN_2$, and that $\cN_2$ and $\cM_2$ are quotient modules of $\cN$ and $\cN_2'$ respectively. Assuming that for some $2\le k<n$ we have constructed $\cN_k'\in \sFPM(R)$ such that $\cM_k$ is a quotient module of $\cN_k'$. Applying Lemma~\ref{L-presentation2} to $(\cN, \cM', \cM)=(\cN_k', \cM_k, \cM_{k+1})$, we find $\cN_{k+1},\cN_{k+1}'\in \sFPM(R)$ such that $\cN_{k+1}'\lesssim \cN_{k+1}$, and that $\cN_{k+1}$ and $\cM_{k+1}$ are quotient modules of $\cN_k'$ and $\cN_{k+1}'$ respectively. In this way we obtain $\cN_k, \cN_k'\in \sFPM(R)$ for all $2\le k\le n$ such that $\cN_k'\lesssim \cN_k$ for all $2\le k\le n$, and that $\cN_{k+1}$ is a quotient module of $\cN_k'$ for all $2\le k <n$, and that $\cN_2$ and $\cM$ are quotient modules of $\cN$ and $\cN_n'$ respectively. Note that for any $\cW_1, \cW_2\in \sFPM(R)$, if $\cW_2$ is a quotient module of $\cW_1$, then $\cW_2\lesssim \cW_1$. Thus we have
$$\cN\gtrsim \cN_2\gtrsim\cN_2'\gtrsim\cN_3\gtrsim\cN_3'\gtrsim\cdots \gtrsim\cN_n\gtrsim\cN_n'\gtrsim\cM.$$
Therefore $\cM\preceq_\rM \cN$. This proves the ``if'' part.
\end{proof}

\begin{remark} \label{R-fp to fg}
From Theorems~\ref{T-fp vs fg} and \ref{T-extension submonoid} we conclude that every state of
$(\rV_\rM(R), \preceq_\rM, \langle {}_RR\rangle_\rM)$ extends to a state of $(\rU_\rM(R), \preceq_\rM^{\rm fg}, \langle {}_RR\rangle_\rM^{\rm fg})$.
Equivalently, every Sylvester module rank function $\dim$ for $R$ extends to a function $\dim': \sFGM(R)\rightarrow \Rb_{\ge 0}$ satisfying the conditions (1)-(3) of Definition~\ref{D-Sylvester mod}. In fact, there is always a canonical extension. In \cite[Theorem 3.3]{Li21} it was shown that $\dim$ extends uniquely to a {\it bivariant Sylvester module rank function} for $R$ which assigns a value $\dim(\cM_1 \mid \cM_2)\in \Rb_{\ge 0}\cup \{+\infty\}$ for every pair of left $R$-modules $\cM_1\subseteq \cM_2$, satisfies the conditions in \cite[Definition 3.1]{Li21}, and extends $\dim$ in the sense that $\dim(\cM)=\dim(\cM \mid \cM)$ for every $\cM\in \sFPM(R)$. Setting $\dim'(\cM)=\dim(\cM \mid \cM)$ for every $\cM\in \sFGM(R)$ provides a canonical extension of $\dim$.
\end{remark}

%%%%%%%%%%%%%%%%%%%%%%%%%%%%%%%%%%%%%%%%%%%%%%%%%%%%%%%%%%%%%%%%%%%%%%%%%%%%%%%%%%%%%%%%%%%%%%%%%%%%%%%%%%%%%%%
\appendix

\section{Goodearl-Handelman theorem and Antoine-Perera-Thiel theorem for partially ordered abelian semigroups} \label{S-POAS}

In this appendix we prove Theorem~\ref{T-extension submonoid}, unifying Theorem~\ref{T-GH} of Goodearl-Handelman and Theorem~\ref{T-GH1} of Antoine-Perera-Thiel in the setting of partially ordered abelian semigroups.

We first extend Lemma~\ref{L-APT} to the case of partially ordered abelian semigroups. In fact, the main part of the proof of \cite[Lemma 5.2.3]{APT18} holds for this situation, and we only indicate the modification needed.

\begin{lemma} \label{L-APT2}
Let $(\rW, \preceq, v)$ be a partially ordered abelian semigroup with order-unit. Let $\rW_1$ be a subsemigroup of $\rW$ containing $v$, and let $\varphi$ be a state of
$(\rW_1, \preceq, v)$. Let $a\in \rW$. Set
$$p=\sup\{(\varphi(b)-\varphi(c))/m \mid b, c\in \rW_1, m\in \Nb, b\preceq c+ma\},$$ and
$$q=\inf\{(\varphi(b)-\varphi(c))/m \mid b, c\in \rW_1, m\in \Nb, b\succeq c+ma\}.$$
Denote by $\rW_2$ the subsemigroup of $\rW$ generated by $\rW_1$ and $a$.
Then $-\infty< p\le q< \infty$ and
$$[p, q]=\{\psi(a) \mid \psi \mbox{ is a state of } (\rW_2, \preceq, v) \mbox{ extending } \varphi\}.$$
\end{lemma}
\begin{proof} As convention, we put  $b+0\cdot a=b$ for $b\in \rW_1$.  Set
$$p'=\sup\{(\varphi(b)-\varphi(c))/m \mid b, c\in \rW_1, m\in \Nb, \bar{m}\in \Zb_{\ge 0}, b+\bar{m}a\preceq c+(m+\bar{m})a\},$$
and
$$q'=\inf\{(\varphi(b)-\varphi(c))/m \mid b, c\in \rW_1, m\in \Nb, \bar{m}\in \Zb_{\ge 0}, b+\bar{m}a\succeq c+(m+\bar{m})a\}.$$
Clearly, for every state $\psi$ of $(\rW_2, \preceq, v)$ extending $\varphi$, one has $\psi(a)\in [p, q]$. In the proof of \cite[Lemma 5.2.3.(4)]{APT18}, it was shown that for every $r\in [p', q']$, there is some state $\psi$ of $(\rW_2, \preceq, v)$ extending $\varphi$ such that $\psi(a)=r$. Thus it suffices to show $p=p', q=q'$ and $-\infty<p\le q<\infty$.

Clearly $p\le p'$. Let $ b, c\in \rW_1, m\in \Nb, \bar{m}\in \Zb_{\ge 0}$ such that
$$ b+\bar{m}a\preceq c+(m+\bar{m})a.$$
Via induction it was shown in the proof of \cite[Lemma 5.2.3.(1)]{APT18} that
$$nb+\bar{m}a\preceq nc+(nm+\bar{m})a$$
for all $n\in \Nb$. Since $v$ is an order-unit, one has $v\preceq \bar{m}a+ kv$ for some $k\in \Nb$. Then
$$nb+v\preceq nb+\bar{m}a+kv\preceq nc+kv+(nm+\bar{m})a$$
for all $n\in \Nb$. Thus
$$p\ge \sup_{n\in \Nb}\frac{\varphi(nb+v)-\varphi(nc+kv)}{nm+\bar{m}}\ge \frac{\varphi(b)-\varphi(c)}{m}.$$ It follows that $p\ge p'$, whence $p=p'$. Also, since $v$ is an order-unit, one has $v\preceq a+ lv$ for some $l\in \Nb$. Then
$$p\ge \varphi(v)-\varphi(lv)=1-l>-\infty.$$

It was shown in the proof of \cite[Lemma 5.2.3.(2)]{APT18} that $q=q'$. Since $v$ is an order-unit, one has $a\preceq nv$ for some $n\in \Nb$. Then $v+a\preceq (n+1)v$, and hence
$$q\le \varphi((n+1)v)-\varphi(v)=n<\infty.$$

The inequality $p\le q$ was shown in the proof of  \cite[Lemma 5.2.3.(1)]{APT18}.
\end{proof}

From Lemma~\ref{L-APT2} and Zorn's lemma we obtain the following result immediately, which unifies Theorems~\ref{T-GH} and \ref{T-GH1}. The last sentence was proven earlier by Blackadar and R{\o}rdam \cite[Corollary 2.7]{BR92}, and as Leonel Robert points out to us, also by Fuchssteiner \cite[Theorem 1]{Fuchssteiner}.

\begin{theorem} \label{T-extension submonoid}
Let $(\rW, \preceq, v)$ be a partially ordered abelian semigroup with order-unit.
Let $\rW_1$ be a subsemigroup of $\rW$ containing $v$, and let $\varphi$ be a state of $(\rW_1, \preceq, v)$. Let $a\in \rW$. Set $$p=\sup\{(\varphi(b)-\varphi(c))/m \mid b, c\in \rW_1, m\in \Nb, b\preceq c+ma\},$$ and
$$q=\inf\{(\varphi(b)-\varphi(c))/m \mid b, c\in \rW_1, m\in \Nb, b\succeq c+ma\}.$$
Then $-\infty<p\le q<\infty$, and
$$[p, q]=\{\psi(a) \mid  \psi \mbox{ is a state of } (\rW, \preceq, v) \mbox{ extending } \varphi\}.$$
In particular,  every state of $(\rW_1, \preceq, v)$ extends to a state of $(\rW, \preceq, v)$.
\end{theorem}

\begin{corollary} \label{C-range}
Let $(\rW, \preceq, v)$ be a partially ordered abelian semigroup with order-unit.
Assume that $(n+1)v\not\preceq nv$ for all $n\in \Nb$.
For any $a\in \rW$, setting
$$p=\sup\{(n-k)/m \mid n, k, m\in \Nb, nv\preceq ma+kv\},$$ and
$$q=\inf\{(n-k)/m  \mid n, k, m\in \Nb, nv\succeq ma+kv\},$$
one has $-\infty< p\le q<\infty$, and
$$[p, q]=\{\varphi(a)\mid \varphi \mbox{ is a state of } (\rW, \preceq, v)\}.$$
If furthermore $a\in \rW$ is an order-unit,  then
$$q=\inf\{n/m \mid n, m\in \Nb, nv\succeq ma\},$$ and
$$p=\sup\{n/m \mid n,m \in \Nb, nv\preceq ma\}.$$
\end{corollary}
\begin{proof}
By assumption we have $kv\not\preceq nv$ for all $k, n\in \Nb$ with $k>n$. Thus  $kv\neq nv$ for all distinct $k, n\in \Nb$.
Then $\rW_1:=\{nv \mid  n\in \Nb\}$ is a subsemigroup of $\rW$ containing $v$, and $\varphi: \rW_1\rightarrow \Rb$ defined by $\varphi(nv)=n$ is the unique state of $(\rW_1, \preceq, v)$.
From Theorem~\ref{T-extension submonoid} we obtain $[p, q]=\{\varphi(a) \mid \varphi \mbox{ is a state of } (\rW, \preceq, v)\}$.

Now assume further that $a\in \rW$ is an order-unit. Set
$$q'=\inf\{n/m \mid n, m\in \Nb, nv\succeq ma\},$$ and
$$p'=\sup\{n/m \mid n,m \in \Nb, nv\preceq ma\}.$$
 Clearly $q\le q'$ and $p\ge p'$.

Let $n, k, m\in \Nb$ such that $nv\succeq ma+kv$. Then $nv\succeq ma+kv\succeq kv$, whence $n\ge k$. If
$$(l(n-k)+k)v\succeq lma+kv$$ for some $l\in \Nb$, then
$$((l+1)(n-k)+k)v\succeq (n-k)v+lma+kv= lma+nv\succeq (l+1)ma+kv.$$
It follows that
$$(l(n-k)+k)v\succeq lma+kv\succeq lma$$
for all $l\in \Nb$. Thus
$$q'\le \inf_{l\in \Nb}\frac{l(n-k)+k}{lm}=\frac{n-k}{m}.$$ Consequently, $q'\le q$. Therefore $q'=q$.

Since $a$ is an order-unit, we have $v\preceq ra$ for some $r\in \Nb$. Then $\frac{1}{r}\le p'$.
Let $n, k, m\in \Nb$ such that $nv\preceq ma+kv$. If $n\le k$, then $\frac{n-k}{m}\le 0\le p'$. Thus we may assume $n>k$. If
$$(l(n-k)+k)v\preceq lma+kv$$ for some $l\in \Nb$, then
$$((l+1)(n-k)+k)v\preceq (n-k)v+lma+kv= lma+nv\preceq (l+1)ma+kv.$$
It follows that
$$(l(n-k)+k)v\preceq lma+kv\preceq lma+kra$$ for all $l\in \Nb$. Thus
$$p'\ge \sup_{l\in \Nb}\frac{l(n-k)+k}{lm+kr}\ge \frac{n-k}{m}.$$ Consequently, $p'\ge p$. Therefore $p'=p$.
\end{proof}

\begin{remark} \label{R-Cuntz}
Let $\cA$ be a stably finite simple unital $C^*$-algebra. In \cite{Cuntz78} Cuntz gave some bounds for the value of any nonzero $x\in M(\cA)$ under dimension functions. Putting
$$s_n(x)=\sup \{s\in \Nb \mid s\langle 1\rangle_\rC\preceq_\rC n\langle x\rangle_\rC\},$$ and
$$r_n(x)=\inf \{r\in \Nb \mid r\langle 1\rangle_\rC\succeq_\rC n\langle x\rangle_\rC\},$$
the limits
$$A_-(x)=\lim_{n\to \infty} s_n(x)/n$$ and
$$A_+(x)=\lim_{n\to \infty} r_n(x)/n$$ exist.
Using Theorem~\ref{T-GH}, in \cite[page 152]{Cuntz78} Cuntz observed that
$$[A_-(x), A_+(x)]\supseteq \{\varphi(x) \mid \varphi \mbox{ is a dimension function for } \cA\}.$$
Since $\cA$ is stably finite, $(n+1)\langle 1\rangle_\rC\not\preceq n\langle 1\rangle_\rC$ for all $n\in \Nb$ \cite[Lemma 4.1]{Cuntz78}. As $\cA$ is simple, $\langle x\rangle_\rC$ is an order-unit of $\rW_\rC(\cA)$ for every nonzero $x\in M(\cA)$. Thus from Corollary~\ref{C-range} we actually have
$$[A_-(x), A_+(x)]= \{\varphi(x) \mid \varphi \mbox{ is a dimension function for } \cA\}$$
for every nonzero $x\in M(\cA)$.
\end{remark}

%%%%%%%%%%%%%%%%%%%%%%%%%%%%%%%%%%%%%%%%%%%%%%%%%%%%%%%%%%%%%%%%%%%%%%%%%%%%%%%%%%%%%%%%%%%%%%%%%%%%%%%%%%%%%%%%%%%%%%%%%%%

\end{document}